\documentclass[11pt, final]{article}
\usepackage{a4}
\usepackage{amsmath}%
\usepackage{amstext}%
\usepackage{amssymb}%
\usepackage{showkeys}%
\usepackage{epsfig}%
\usepackage{cite}
\usepackage{tikz}
\usepackage{subfig}
\usepackage{caption}
\usepackage{multirow}
\usepackage{booktabs}
\usepackage{floatrow}

\usepackage{listings}
\usepackage{geometry}

\newtheorem{theorem}{Theorem}

\newtheorem{corollary}[theorem]{Corollary}

\newtheorem{lemma}[theorem]{Lemma}

\newtheorem{pb}[theorem]{Problem}
\newenvironment{problem}{\pb\rm}{\endpb}

\newtheorem{proposition}[theorem]{Proposition}
\newtheorem{rem}[theorem]{Remark}
\newenvironment{remark}{\rem\rm}{\endrem}

\newcounter{unnumber}

\newtheorem{prf}[unnumber]{Proof.}
\newenvironment{proof}{\prf\rm}{\hfill{$\blacksquare$}\endprf}
\newcommand{\R}{\mathbb{R}}%
\newcommand{\ol}{\overline}%

\renewcommand{\>}{\right\rangle}
\newcommand{\<}{\left\langle}

\newcommand{\bx}{\ensuremath{\overline{x}}}

\newcommand{\bv}{\ensuremath{\overline{v}}}

\newcommand{\bp}{\ensuremath{\overline{p}}}

\renewcommand{\bv}{\ensuremath{\overline{v}}}

\newcommand{\f}{\ensuremath{\boldsymbol}}
\newcommand{\fx}{\ensuremath{\boldsymbol{x}}}

\newcommand{\fbx}{\ensuremath{\boldsymbol{\overline{x}}}}

\newcommand{\fy}{\ensuremath{\boldsymbol{y}}}
\newcommand{\fz}{\ensuremath{\boldsymbol{z}}}

\newcommand{\fK}{\ensuremath{\boldsymbol{\mathcal{K}}}}
\newcommand{\h}{\ensuremath{\mathcal{H}}}
\newcommand{\g}{\ensuremath{\mathcal{G}}}

\DeclareMathOperator*\inte{int}%
\DeclareMathOperator*\sqri{sqri}%
\DeclareMathOperator*\ri{ri}%
\DeclareMathOperator*\dom{dom}%
\DeclareMathOperator*\B{\overline{\R}}%
\DeclareMathOperator*\gr{Gr}%
\DeclareMathOperator*\ran{ran}%
\DeclareMathOperator*\id{Id}%
\DeclareMathOperator*\prox{prox}%
\DeclareMathOperator*\argmin{argmin}

\DeclareMathOperator*\zer{zer}

\DeclareMathOperator*\fix{Fix}
\DeclareMathOperator*\proj{proj}

\title{Inducing strong convergence into the asymptotic behaviour of proximal splitting algorithms in Hilbert spaces}

\author{Radu Ioan Bo\c{t} \thanks{University of Vienna, Faculty of Mathematics, Oskar-Morgenstern-Platz 1, A-1090 Vienna, Austria,
email: radu.bot@univie.ac.at. Research partially supported by FWF (Austrian Science Fund), project I 2419-N32.}   \and
Ern\"{o} Robert Csetnek \thanks {University of Vienna, Faculty of Mathematics, Oskar-Morgenstern-Platz 1, A-1090 Vienna, Austria,
email: ernoe.robert.csetnek@univie.ac.at. Research supported by FWF (Austrian Science Fund), projects M 1682-N25 and P 29809-N32.}\and 
Dennis Meier \thanks {University of Vienna, Faculty of Mathematics, Oskar-Morgenstern-Platz 1, A-1090 Vienna, Austria,
email: meierd61@univie.ac.at. Research partially supported by FWF (Austrian Science Fund), project I 2419-N32, and by the Doctoral Programme {\it Vienna Graduate School on Computational Optimization (VGSCO)}, project W1260-N35.}}

\begin{document}
\maketitle

\noindent \textbf{Abstract.} Proximal splitting algorithms for monotone inclusions (and convex optimization problems) in Hilbert spaces share the common 
feature to guarantee  for the generated sequences in general weak convergence to a solution. In order to achieve strong convergence, one usually needs 
to impose more restrictive properties for the involved operators, like strong monotonicity (respectively, strong convexity for optimization problems). 
In this paper, we propose a modified Krasnosel'ski\u{\i}--Mann algorithm in connection with the determination of a fixed point of a nonexpansive mapping 
and show strong convergence of the iteratively generated sequence to the minimal norm solution of the problem. Relying on this, we derive a forward-backward 
and a Douglas-Rachford algorithm, both endowed with Tikhonov regularization terms, which generate iterates that strongly converge to the minimal norm 
solution of the set of zeros of the sum of two maximally monotone operators. Furthermore, we formulate strong convergent primal-dual algorithms of forward-backward and Douglas-Rachford-type for
highly structured monotone inclusion problems involving parallel-sums and compositions with linear operators. The resulting iterative schemes are 
particularized to  the solving of convex minimization problems. The theoretical results are illustrated by numerical experiments on the split feasibility problem 
in infinite dimensional spaces. 
\vspace{1ex}

\noindent \textbf{Key Words.} fixed points of nonexpansive mappings, Tikhonov regularization, splitting methods, forward-backward algorithm, 
Douglas-Rachford algorithm, primal-dual algorithm\vspace{1ex}

\noindent \textbf{AMS subject classification.} 47J25, 47H09, 47H05, 90C25

\section{Introduction and preliminaries}\label{sec1}

Let ${\cal H}$ be a real Hilbert space endowed with inner product $\langle\cdot,\cdot\rangle$ and associated norm
$\|\cdot\|=\sqrt{\langle \cdot,\cdot\rangle}$. Let $T:{\cal H}\rightarrow {\cal H}$ be a nonexpansive mapping, that is 
$\|Tx-Ty\|\leq\|x-y\|$ for all $x,y\in {\cal H}$. One of the most popular iterative methods for finding a fixed point of the operator $T$ 
is the Krasnosel'ski\u{\i}--Mann algorithm
\begin{equation}\label{km-alg}x_{n+1}=x_n+\lambda_n\big(T x_n-x_n\big) \ \forall n \geq 0,\end{equation}
where $x_0\in {\cal H}$ is arbitrary and $(\lambda_n)_{n \geq 0}$ is a sequence of nonnegative real numbers. Provided 
$\fix T=\{x\in {\cal H}:Tx=x\}\neq\emptyset$, one can show under mild conditions imposed on $(\lambda_n)_{n\geq 0}$, 
that the sequence $(x_n)_{n\geq 0}$ converges weakly to an element in $\fix T$ (see for instance \cite{bauschke-book}). 

The applications and the impact of this fundamental result go beyond the usual fixed point theory, representing in fact the starting point for the derivation of algorithms and related convergence statements in connection with the solving of monotone inclusions. In this context, we mention the classical \emph{forward-backward algorithm} for determining a zero of the sum of a set-valued maximally monotone operator and a  single-valued and cocoercive one and the \emph{Douglas-Rachford algorithm} for determining a zero of the sum of two set-valued maximally monotone operators. The paradigms behind these classical methods can be transferred to the solving of convex optimization problems, too (see \cite{bauschke-book}).

The iterative algorithms mentioned above share the common property that the generated sequences converge weakly to a solution of the problem under investigation. However, for applications where infinite dimensional functional spaces are involved, weak convergence is not satisfactory. 
In order to achieve strong convergence, one usually needs to impose more restrictive properties for the involved operators, like strong monotonicity when considering monotone inclusions and strong convexity when solving optimization problems. Since there evidently are applications for which these stronger properties are not fulfilled, the interest of the applied mathematics community in developing algorithms which generate iterates that strongly convergence is justified. 
 
We mention in this sense the \emph{Halpern algorithm}  and its numerous variants designed for finding a fixed point of a nonexpansive mapping (see for instance \cite{chidume2009}). In the context of solving
monotone inclusions, we mention the \emph{proximal-Tikhonov algorithm} 
$$x_{n+1}=\big(\id+\lambda_n(A+\mu_n\id)\big)^{-1}(x_n) \ \forall n \geq 0,$$ 
where $A:{\cal H}\rightrightarrows{\cal H}$ is a maximally monotone operator, $\id$ is the identity operator on ${\cal H}$ and 
$(\lambda_n)_{n \geq 0}$ and $(\mu_n)_{n \geq 0}$ are sequences of nonnegative real numbers. Under mild conditions imposed on 
$(\lambda_n)_{n\geq 0}$ and $(\mu_n)_{n\geq 0}$, one can prove strong convergence of $(x_n)_{n\geq 0}$ to the minimal norm solution 
of the set of zeros of $A$ (see \cite{lehdili-moudafi, xu2002}). It is important to emphasize the Tikhonov regularization terms  
$(\mu_n\id)_{n\geq 0}$ in the above scheme, which actually enforces the strong convergence property. In the absence of the regularization term, 
the above numerical scheme becomes the classical \emph{proximal algorithm} for determining a zero of the operator $A$, 
for which in general only weak convergence can be proved (see \cite{rock-prox}). 
For more theoretical results concerning Tikhonov regularization  and more motivational arguments for using such techniques, 
especially for optimization problems, we refer the reader to Attouch's paper \cite{attouch1996}. For other techniques and tools in order to 
achieve strong convergence we mention also the works of Haugazeau \cite {haug} and \cite{ba-co}. 

In this article, we will first introduce and investigate a modified Krasnosel'ski\u{\i}--Mann algorithm with relaxation parameters, having the outstanding property that it generates a sequence of iterates which converges strongly to the minimal norm solution 
of the fixed points set of a nonexpansive mapping. In contrast to \cite{moudafi2000, xu2002, chidume2009} (see also the  references therein), where the techniques and tools used have their roots in fixed point theory results for contractions, our convergence statements follow more directly.  
Relying on this, we derive a forward-backward and a Douglas-Rachford algorithm, both endowed with Tikhonov regularization terms, which generate iterates that strongly converge to the minimal norm solution of the set of zeros of the sum of two maximally monotone operators. 
The resulting iterative schemes are particularized to the minimization of the sum of two convex functions.

Furthermore, we deal with complexly structured monotone inclusions where parallel-sums and compositions with linear operators are involved. By making use of
modern primal-dual techniques (see \cite{br-combettes, combettes-pesquet, Con12, vu, ch-pck} and also \cite{b-c-h2, b-h2}), we derive strongly convergent numerical schemes of 
forward-backward and Douglas-Rachford type, both involving Tikhonov regularization terms and having the remarkable property that all the operators are evaluated separately. Moreover, the designed algorithms solve both the structured monotone inclusion problem 
and its dual monotone inclusion problem in the sense of Attouch-Th\'{e}ra (see \cite{AttThe96}). When particularized to convex 
optimization problems, this means the concomitantly solving of a primal problem and its Fenchel dual 
one. For other types of primal-dual algorithms with strong convergence properties and their applications we refer the 
reader to \cite{abd-comb-sha2015, att-br-comb2016}. Finally, in the last section we carry out numerical experiments 
on the split feasibility problem in infinite dimensional Hilbert spaces which illustrate the potential of the algorithm endowed with Tikhonov regularization terms. 

In the remaining of this section, we recall some results which will play a decisive role in the convergence analysis of the proposed algorithms. The following result  is related to the convergence of a sequence satisfying a sharp quasi-Fej\'er monotonicity property and follows as a direct consequence of \cite[Lemma 2.5]{xu2002}.

\begin{lemma}\label{ineq-s-n+1-s-n} Let $(a_n)_{n\geq 0}$ be a sequence of non-negative real numbers 
satisfying the inequality $$a_{n+1}\leq(1-\theta_n)a_n+ \theta_n b_n + \varepsilon_n \ \forall n\geq 0,$$
where 

(i) $0\leq\theta_n\leq 1$ for all $n\geq 0$ and $\sum_{n\geq 0}\theta_n=+\infty$; 

(ii) $\limsup_{n \rightarrow + \infty} b_n \leq 0$;

(iii) $\varepsilon_n\geq 0$ for all $n\geq 0$ and $\sum_{n\geq 0}\varepsilon_n<+\infty$. 

Then the sequence $(a_n)_{n\geq 0}$ converges to $0$. 
\end{lemma}

We close this section with a result that is a consequence of the demiclosedness principle (see \cite[Corollary 4.18]{bauschke-book}) and it will be used in the proof of Theorem \ref{tikhonov-nonexp}, which is the main result of this paper. 

\begin{lemma}\label{demi} Let $T:{\cal H}\rightarrow {\cal H}$ be a nonexpansive operator and let $(x_n)_{n\geq 0}$ be a sequence in 
${\cal H}$ and $x\in {\cal H}$ be such that $w-\lim_{n\rightarrow+\infty} x_n=x$ and 
$(Tx_n-x_n)_{n\geq 0}$ converges strongly to $0$ as $n\rightarrow+\infty$. Then $x\in\fix T$.
\end{lemma}

\section{A strongly convergent Krasnosel'ski\u{\i}--Mann algorithm}\label{sec2}

In order to induce strong convergence into the asymptotic behaviour of the Krasnosel'ski\u{\i}--Mann algorithm for determining a fixed point of a nonexpansive mapping $T: {\cal H}\rightarrow{\cal H}$, we propose the following modified version of it:
\begin{equation}\label{it-sch-nonexp}x_{n+1}=\beta_nx_n+\lambda_n\big(T(\beta_nx_n)-\beta_nx_n\big) \ \forall n\geq 0,\end{equation}
where $x_0\in{\cal H}$ is the starting point and $(\lambda_n)_{n \geq 0}$ and $(\beta_n)_{n \geq 0}$ suitably chosen sequences of positive numbers. 
In the proof of the theorem below, we denote by $\proj_C:{\cal H}\rightarrow C, \proj_C(x) = \argmin\limits_{c \in C}\|x-c\|$, 
the projection operator onto the nonempty closed convex set $C\subseteq {\cal H}$. We notice that for a nonexpansive mapping 
$T: {\cal H}\rightarrow{\cal H}$, its set of fixed points $\fix T$ is closed and convex (see \cite[Corollary 4.15]{bauschke-book}). 

\begin{theorem}\label{tikhonov-nonexp} Let $(\lambda_n)_{n\geq 0}$ and $(\beta_n)_{n\geq 0}$ be real sequences satisfying the conditions: 

(i) $0<\beta_n\leq 1$ for any $n \geq 0$, $\lim_{n\rightarrow+\infty}\beta_n=1$, $\sum_{n\geq 0} (1-\beta_n)=+\infty$ and 
$\sum_{n\geq 1} |\beta_n-\beta_{n-1}|<+\infty$; 

(ii) $0<\lambda_n\leq 1$ for any $n \geq 0$, $\liminf_{n\rightarrow+\infty}\lambda_n>0$ and 
$\sum_{n\geq 1} |\lambda_n-\lambda_{n-1}|<+\infty$.

Consider the iterative scheme \eqref{it-sch-nonexp} with and arbitrary starting point $x_0\in{\cal H}$ and 
a nonexpansive mapping $T: {\cal H}\rightarrow{\cal H}$ fulfilling $\fix T\neq\emptyset$. Then $(x_n)_{n \geq 0}$ converges strongly to 
$\proj_{\fix T}(0)$.
\end{theorem}

\begin{proof} For the beginning, we prove that $(x_n)_{n\geq 0}$ is bounded. Let $x\in\fix T$. Due to the nonexpansiveness of $T$, 
we have for any $n\geq 0$: 
\begin{align*}\|x_{n+1}- x\|= & \ \|(1-\lambda_n)(\beta_n x_n- x)+\lambda_n\big(T(\beta_nx_n)-T x\big)\|\\
\leq &  \ (1-\lambda_n)\|\beta_nx_n- x\|+\lambda_n\|T(\beta_nx_n)-T x\|\\
\leq & \ \|\beta_nx_n- x\|\\
 = & \ \|\beta_n(x_n- x)+(\beta_n-1) x\|\\
\leq & \ \beta_n\|x_n- x\|+(1-\beta_n)\| x\|.
\end{align*}

A simple induction leads to the inequality $$\|x_n- x\|\leq \max\{\|x_0- x\|,\| x\|\} \ \forall n\geq 0,$$
hence $(x_n)_{n\geq 0}$ is bounded. 

We claim that \begin{equation}\label{x-k+1-xk}\|x_{n+1}-x_n\|\rightarrow 0\mbox{ as }n\rightarrow+\infty.\end{equation}
Indeed, by taking into account that $T$ is nonexpansive and that $(x_n)_{n\geq 0}$ is bounded, we obtain for any $n \geq 1$ the following estimates:
\begin{align*}
\|x_{n+1}-x_n\| = & \ \|(1-\lambda_n)\beta_nx_n-(1-\lambda_{n-1})\beta_{n-1}x_{n-1}+\lambda_nT(\beta_nx_n)-\lambda_{n-1}
T(\beta_{n-1}x_{n-1})\|\\
\leq & \ \|(1-\lambda_n)(\beta_nx_n-\beta_{n-1}x_{n-1})+(\lambda_{n-1}-\lambda_n)\beta_{n-1}x_{n-1} \|\\
& + \| \lambda_n\big(T(\beta_nx_n) - T(\beta_{n-1}x_{n-1})\big)+(\lambda_n-\lambda_{n-1})T(\beta_{n-1}x_{n-1})\|\\
\leq & \ \|\beta_nx_n-\beta_{n-1}x_{n-1}\|+|\lambda_n-\lambda_{n-1}|C_1,
\end{align*}
where $C_1>0$. 

Further, we derive for any $n \geq 1$:
\begin{align*}\|x_{n+1}-x_n\|\leq & \ \|\beta_n(x_n-x_{n-1})+(\beta_n-\beta_{n-1})x_{n-1}\|+|\lambda_n-\lambda_{n-1}|C_1\\
\leq & \ \beta_n\|x_n-x_{n-1}\|+|\beta_n-\beta_{n-1}|C_2+|\lambda_n-\lambda_{n-1}|C_1,
\end{align*}
where $C_2>0$. Statement \eqref{x-k+1-xk} is a consequence of Lemma \ref{ineq-s-n+1-s-n}, for 
$a_n:= \|x_n-x_{n-1}\|, b_n:=0, \varepsilon_n:= |\beta_n-\beta_{n-1}|C_2+|\lambda_n-\lambda_{n-1}|C_1$ and $\theta_n:=1-\beta_n, n \geq 1$. 

In the following we prove that \begin{equation}\label{xk-txk}\|x_n-Tx_n\|\rightarrow 0\mbox{ as }n\rightarrow+\infty.\end{equation}
For any $n \geq 0$ we have the following inequalities: 
\begin{align*} \|x_n-Tx_n\|\leq & \ \|x_{n+1}-x_n\|+\|x_{n+1}-Tx_n\|\\
= & \ \|x_{n+1}-x_n\|+\|(1-\lambda_n)(\beta_nx_n-Tx_n)+\lambda_n\big(T(\beta_nx_n)-Tx_n\big)\|\\
\leq & \ \|x_{n+1}-x_n\|+(1-\lambda_n)\|\beta_nx_n-Tx_n\|+\lambda_n\|\beta_nx_n-x_n\|\\
\leq & \ \|x_{n+1}-x_n\|+ (1-\lambda_n)\|\beta_nx_n-\beta_nTx_n\|\\
& +(1-\lambda_n)\|\beta_nTx_n-Tx_n\|+\lambda_n(1-\beta_n)\|x_n\|\\
\leq & \ \|x_{n+1}-x_n\|+ (1-\lambda_n)\|x_n-Tx_n\|\\
& +(1-\lambda_n)(1-\beta_n)\|Tx_n\|+\lambda_n(1-\beta_n)\|x_n\|.
\end{align*}
From here we deduce that for any $n \geq 0$: 
$$\lambda_n\|x_n-Tx_n\|\leq \|x_{n+1}-x_n\|+ (1-\lambda_n)(1-\beta_n)\|Tx_n\|+\lambda_n(1-\beta_n)\|x_n\|.$$
Taking into account that $(x_n)_{n\geq 0}$ is bounded, $\eqref{x-k+1-xk}$ and the properties of the sequences involved, 
we derive from the last inequality that \eqref{xk-txk} holds.  

In what follows we show that $(x_n)_{n\geq 0}$ actually converges strongly to $\proj_{\fix T}(0):=\ol x$. 
Since $T$ is nonexpansive, we have for any $n \geq 0$:
\begin{align*}\|x_{n+1}- \ol x\|= & \ \|(1-\lambda_n)(\beta_n x_n- \ol x)+\lambda_n\big(T(\beta_nx_n)-T \ol x\big)\|\\
\leq &  \ (1-\lambda_n)\|\beta_nx_n- \ol x\|+\lambda_n\|T(\beta_nx_n)-T\ol  x\|\\
\leq & \ \|\beta_nx_n- \ol x\|.
\end{align*}
Hence,
\begin{align}\label{ineq-str-conv} \|x_{n+1}- \ol x\|^2 \leq & \ \|\beta_nx_n- \ol x\|^2\nonumber\\    
= & \ \|\beta_n(x_n-\ol x)+(\beta_n-1)\ol x\|^2\nonumber\\
= & \ \beta_n^2\|x_n-\ol x\|^2+2\beta_n(1-\beta_n)\langle -\ol x,x_n-\ol x\rangle+(1-\beta_n)^2\|\ol x\|^2\nonumber\\
\leq & \ \beta_n\|x_n-\ol x\|^2+(1-\beta_n)\big(2\beta_n\langle -\ol x,x_n-\ol x\rangle+(1-\beta_n)\|\ol x\|^2\big) \ \forall n \geq 0.
\end{align}
Next we show that
 \begin{equation}\label{limsup}\limsup_{n\rightarrow+\infty}\langle -\ol x,x_n-\ol x\rangle\leq 0.\end{equation}
Assuming the contrary, there would exist a positive real number $l$  and a subsequence $(x_{k_j})_{j\geq 0}$
such that 
$$\langle -\ol x,x_{k_j}-\ol x\rangle\geq l  > 0 \ \forall j \geq 0.$$
Due to the boundedness of the sequence  $(x_n)_{n\geq 0}$, we can assume without losing the generality that 
$(x_{k_j})_{j\geq 0}$ weakly converges  to an element $y \in {\cal H}$. According to Lemma \ref{demi}, by taking into consideration \eqref{xk-txk}, it follows that $y\in\fix T$. From this and the 
variational characterization of the projection we easily derive 
$$\lim_{j\rightarrow+\infty}\langle -\ol x,x_{k_j}-\ol x\rangle=\langle -\ol x,y-\ol x\rangle\leq 0,$$
which leads to a contradiction. This shows that \eqref{limsup} holds. 
Thus $$\limsup_{n\rightarrow+\infty}\big(2\beta_n\langle -\ol x,x_n-\ol x\rangle+(1-\beta_n)\|\ol x\|^2\big)\leq 0.$$
A direct application of Lemma \ref{ineq-s-n+1-s-n} to \eqref{ineq-str-conv}, for 
$a_n:= \|x_n-\ol x\|^2, b_n:=2\beta_n\langle -\ol x,x_n-\ol x\rangle+(1-\beta_n)\|\ol x\|^2, \varepsilon_n:= 0$ and 
$\theta_n:=1-\beta_n, n \geq 0,$ delivers the desired conclusion. 
\end{proof}

\begin{remark}\label{parameters} \rm Condition (ii) in the previous theorem is satisfied by every monotonically increasing (and, 
in consequence, convergent) sequence $(\lambda_n)_{n \geq 0} \subseteq (0,1]$ and also by every monotonically decreasing (and, in 
consequence, convergent) sequence $(\lambda_n)_{n \geq 0} \subseteq (0,1]$ having as limit a positive number.

Condition (i) in the previous theorem is satisfied by every monotonically increasing sequence 
$(\beta_n)_{n \geq 0} \subseteq (0,1]$ which fulfills $\lim_{n\rightarrow+\infty}\beta_n=1$ and $\sum_{n\geq 0} (1-\beta_n)=+\infty$, 
as it is for instance the sequence with $\beta_0 \in \left(0, \frac{1}{2}\right)$ and $\beta_n = 1-\frac{1}{n+1}$ for any $n \geq 1$.
\end{remark}

An immediate consequence of Theorem \ref{tikhonov-nonexp} is the following corollary, which proposes an iterative scheme that finds a minimal norm solution of the set of fixed points of an averaged operator. Let $\alpha\in(0,1)$ be fixed. 
We say that $R:{\cal H}\rightarrow{\cal H}$ is an $\alpha$-averaged operator if there exists a nonexpansive operator $T:{\cal H}\rightarrow{\cal H}$ such that $R=(1-\alpha)\id+\alpha T$. It is obvious that $\alpha$-averaged operators are also nonexpansive. 
The $\frac{1}{2}$-averaged operators are nothing else than the firmly nonexpansive ones and form the most most important representatives of this class. For properties and insights into these families  of operators we refer the reader to \cite{bauschke-book}. The following result will play in the next section a determinant role in the converge analysis of the forward-backward method endowed with Tikhonov regularization term.

\begin{corollary}\label{tikhonov-av} Consider the iterative scheme 
\begin{equation}\label{it-sch-av}x_{n+1}=\beta_nx_n+\lambda_n\big(R(\beta_nx_n)-\beta_nx_n\big) \ \forall n\geq 0,\end{equation} 
with $x_0\in{\cal H}$ as starting point, $R: {\cal H}\rightarrow{\cal H}$ an $\alpha$-averaged operator, for $\alpha\in(0,1)$, 
such that $\fix R\neq\emptyset$ and $(\lambda_n)_{n\geq 0}$ and $(\beta_n)_{n\geq 0}$ real sequences satisfying  the conditions: 

(i) $0<\beta_n\leq 1$ for any $n\geq 0$, $\lim_{n\rightarrow+\infty}\beta_n=1$, $\sum_{n\geq 0} (1-\beta_n)=+\infty$ and 
$\sum_{n \geq 1} |\beta_n-\beta_{n-1}|<+\infty$; 

(ii) $0<\lambda_n\leq 1/{\alpha}$ for any $n\geq 0$, $\liminf_{n\rightarrow+\infty}\lambda_n>0$ and 
$\sum_{n\geq 1} |\lambda_n-\lambda_{n-1}|<+\infty$.
Then $(x_n)_{n\geq 0}$ converges strongly to $\proj_{\fix R}(0)$.
\end{corollary}

\begin{proof} Since $R$ is $\alpha$-averaged, there exists a nonexpansive operator 
$T:{\cal H}\rightarrow{\cal H}$ such that $R=(1-\alpha)\id+\alpha T$. The conclusion follows from Theorem \ref{tikhonov-nonexp}, by taking into account that \eqref{it-sch-av} is equivalent to 
$$x_{n+1}=\beta_nx_n+\alpha\lambda_n\big(T(\beta_nx_n)-\beta_nx_n\big) \ \forall n\geq 0$$ and that $\fix R=\fix T$.  
\end{proof}

\section{A forward-backward algorithm with Tikhonov regularization term}\label{sec3}

This section is dedicated to the formulation and convergence analysis of a forward-backward algorithm with Tikhonov regularization terms, which generates a sequence of iterates that converges strongly to the minimal norm solution of the set of zeros of the sum of two maximally monotone operators, one of them being single-valued. 

For readers' convenience, we recall some standard notions and results in monotone operator theory 
which will be used in the following (see also \cite{bo-van, bauschke-book, simons}). For an arbitrary set-valued operator $A:{\cal H}\rightrightarrows {\cal H}$ we denote by 
$\gr A=\{(x,u)\in {\cal H}\times {\cal H}:u\in Ax\}$ its graph. Then $A^{-1} :{\cal H}\rightrightarrows {\cal H}$, which is the operator with $\gr A^{-1}=\{(x,u)\in {\cal H}\times {\cal H}:x\in Au\}$, denotes the inverse operator of $A$. 
We use also the notation $\zer A=\{x\in{\cal{H}}:0\in Ax\}$ for the set of zeros of $A$. We say that $A$ is monotone, if $\langle x-y,u-v\rangle\geq 0$ for all $(x,u),(y,v)\in\gr A$. A monotone operator $A$ is said to be maximally monotone, if there exists no proper monotone extension of the graph of $A$ on ${\cal H}\times {\cal H}$.
The resolvent of $A$, $J_A:{\cal H} \rightrightarrows {\cal H}$, is defined by $$J_A=(\id+A)^{-1},$$ where $\id :{\cal H} \rightarrow {\cal H}, \id(x) = x$ for all $x \in {\cal H}$, is the identity operator on ${\cal H}$. Moreover, if $A$ is maximally monotone, then $J_A:{\cal H} \rightarrow {\cal H}$ is single-valued and maximally monotone
(see \cite[Proposition 23.7 and Corollary 23.10]{bauschke-book}). 

Let $\gamma>0$. We say that $B:{\cal H}\rightarrow {\cal H}$ is $\gamma$-cocoercive, if $\langle x-y,Bx-By\rangle\geq \gamma\|Bx-By\|^2$ for all $x,y\in {\cal H}$. 

The following technical result (see \cite[Theorem 3(b)]{og-yam} and \cite[Proposition 2.4]{comb-yam})  gives an expression for the the averaged parameter of the composition of two averaged operators. We refer  also
to \cite[Proposition 4.32]{bauschke-book} for other results of this type.

\begin{proposition}\label{av-og-yam} 
Let $T_i:{\cal H}\rightarrow {\cal H}$ be $\alpha_i$-averaged, where $\alpha_i\in(0,1)$, $i=1,2$. 
Then the composition $T_1\circ T_2$ is $\alpha$-averaged, where 
$$\alpha=\frac{\alpha_1+\alpha_2-2\alpha_1\alpha_2}{1-\alpha_1\alpha_2}\in(0,1).$$ 
\end{proposition}

\begin{theorem}\label{tikhonov-fb-operators} Let $A:{\cal H}\rightrightarrows {\cal H}$ be a maximally monotone operator and $B:{\cal H}\rightarrow {\cal H}$ a $\beta$-cocoercive operator, for $\beta>0$, such that $\zer(A+B)\neq\emptyset$. Let $\gamma\in(0,2\beta]$. 
Consider the iterative scheme 
\begin{equation}\label{it-sch-fb-operators}
x_{n+1}=(1-\lambda_n)\beta_nx_n+\lambda_n J_{\gamma A}\big(\beta_nx_n-\gamma B(\beta_nx_n)\big) \ \forall n\geq 0,\end{equation} 
with $x_0\in{\cal H}$ as starting point and $(\lambda_n)_{n\geq 0}$ and $(\beta_n)_{n\geq 0}$ real sequences satisfying  the conditions: 

(i) $0<\beta_n\leq 1$ for any $n\geq 0$, $\lim_{n\rightarrow+\infty}\beta_n=1$, $\sum_{n\geq 0} (1-\beta_n)=+\infty$ and 
$\sum_{n \geq 1} |\beta_n-\beta_{n-1}|<+\infty$; 

(ii) $0<\lambda_n\leq \frac{4\beta-\gamma}{2\beta}$ for any $n\geq 0$, $\liminf_{n\rightarrow+\infty}\lambda_n>0$ and 
$\sum_{n\geq 1} |\lambda_n-\lambda_{n-1}|<+\infty$.

Then $(x_n)_{n \geq 0}$ converges strongly to $\proj_{\zer(A+B)}(0)$.
\end{theorem}

\begin{proof} It is immediate that the iterative scheme \eqref{it-sch-fb-operators} can be written in the form 

$$x_{n+1}=\beta_nx_n+\lambda_n\big(T(\beta_nx_n)-\beta_nx_n\big) \ \forall n\geq 0,$$
where $T=J_{\gamma A}\circ(\id -\gamma B)$.

We consider two cases. The first one is when $\gamma\in(0,2\beta)$. 

According to \cite[Corollary 23.8 and Remark 4.24(iii)]{bauschke-book}, 
$J_{\gamma A}$ is $\frac{1}{2}$-cocoercive. 
Moreover, by \cite[Proposition 4.33]{bauschke-book}, $\id -\gamma B$ is $\frac{\gamma}{2\beta}$-averaged. Combining this with 
Proposition \ref{av-og-yam}, we derive that $T$ is $\frac{2\beta}{4\beta-\gamma}$-averaged. The statement follows now from 
Corollary \ref{tikhonov-av}, by noticing that $\fix T=\zer(A+B)$ (see \cite[Proposition 25.1(iv)]{bauschke-book}). 

The second case is when $\gamma=2\beta$. The cocoercivity of $B$ implies that $\id-\gamma B$ is nonexpansive, hence the operator 
$T = J_{\gamma A}\circ(\id -\gamma B)$ is nonexpansive, too, the conclusion follows in this situation from Theorem \ref{tikhonov-nonexp}. 
\end{proof}

\begin{remark}\label{tikhonov} \rm The choice $\lambda_n=1$ for any $n \geq 0$ in the previous theorem leads to the iterative scheme 
$$x_{n+1}= J_{\gamma A}\big(\beta_nx_n-\gamma B(\beta_nx_n)\big) \ \forall n\geq 0,$$
which further becomes in case $B=0$
$$x_{n+1}= J_{\gamma A}\big(\beta_nx_n\big) \ \forall n\geq 0.$$
This last relation can be equivalently written as 
$$x_n\in \frac{1}{\beta_n}x_{n+1}+\frac{\gamma}{\beta_n}Ax_{n+1}=\left(\id +\varepsilon_n\id+\frac{\gamma}{\beta_n}A\right)(x_{n+1}),$$
where $\varepsilon_n\id$ (with $\varepsilon_n:= \frac{1}{\beta_n}-1>0$ and $\lim_{n\rightarrow+\infty}\varepsilon_n=0$) represents 
the Tikhonov regularization term, which enforces the strong convergence of the sequence $(x_n)_{n\geq 0}$ to the minimal norm solution.
For other types of Tikhonov-like methods for monotone inclusion problems we refer the reader to \cite{xu2002, sahu-ansari-yao2015, lehdili-moudafi} and the references therein. 
\end{remark}

In the remaining of this section we turn our attention to the solving of optimization problems of the form
\begin{equation}\label{pr-opt-fb}
\min_{x \in {\cal H}}\{ f(x) + g(x)\},
\end{equation}
where $f:{\cal H}\rightarrow\R\cup\{+\infty\}$ is a proper, convex and lower semicontinuous function and 
$g:{\cal H}\rightarrow \R$ is a convex and Fr\'{e}chet differentiable function with $\frac{1}{\beta}$-Lipschitz continuous gradient, for $\beta > 0$.

For a proper, convex and lower semicontinuous function $f:{\cal H}\rightarrow\R\cup\{+\infty\}$, its (convex) subdifferential at $x\in {\cal H}$ is defined as
$$\partial f(x)=\{u\in {\cal H}:f(y)\geq f(x)+\<u,y-x\> \ \forall y\in {\cal H}\},$$ 
for $x \in {\cal H}$ with $f(x) = +\infty$ and as $\partial f(x) = \emptyset$, otherwise. When seen as a set-valued mapping, the convex subdifferential is a 
maximally monotone operator (see \cite{rock}) and its resolvent is given by $J_{\partial f}=\prox_{f}$ (see \cite{bauschke-book}),
where $\prox_{f}:{\cal H}\rightarrow {\cal H}$,
\begin{equation}\label{prox-def}\prox\nolimits_{ f}(x)=\argmin_{y\in {\cal H}}\left \{f(y)+\frac{1}{2}\|y-x\|^2\right\},
\end{equation}
denotes the proximal operator of $f$. 

\begin{corollary}\label{tikhonov-fb-optimization} Let $f:{\cal H}\rightarrow\R\cup\{+\infty\}$ be a proper, convex and lower semicontinuous 
function and $g:{\cal H}\rightarrow \R$ a convex and Fr\'{e}chet differentiable function with $\frac{1}{\beta}$-Lipschitz continuous gradient, for $\beta > 0$, 
such that $\argmin_{x\in {\cal H}}\{f(x)+g(x)\}\neq\emptyset$. Let $\gamma\in(0,2\beta]$. 
Consider the iterative scheme 
\begin{equation}\label{it-sch-fb-optimization}
x_{n+1}=(1-\lambda_n)\beta_nx_n+\lambda_n \prox\nolimits_{\gamma f}\big(\beta_nx_n-\gamma \nabla g(\beta_nx_n)\big) \ \forall n\geq 0,\end{equation} 
with $x_0\in{\cal H}$ as starting point and $(\lambda_n)_{n\geq 0}$ and $(\beta_n)_{n\geq 0}$ real sequences satisfying  the conditions: 

(i) $0<\beta_n\leq 1$ for any $n\geq 0$, $\lim_{n\rightarrow+\infty}\beta_n=1$, $\sum_{n\geq 0} (1-\beta_n)=+\infty$ and 
$\sum_{n \geq 1} |\beta_n-\beta_{n-1}|<+\infty$; 

(ii) $0<\lambda_n\leq \frac{4\beta-\gamma}{2\beta}$ for any $n\geq 0$, $\liminf_{n\rightarrow+\infty}\lambda_n>0$ and 
$\sum_{n\geq 1} |\lambda_n-\lambda_{n-1}|<+\infty$.

Then $(x_n)_{n \geq 0}$ converges strongly to the minimal norm solution of \eqref{pr-opt-fb}. 
\end{corollary}

\begin{proof} The statement is a direct consequence of Theorem \ref{tikhonov-fb-operators},  
by choosing $A:=\partial f$ and $B:=\nabla g$ and by taking into account that
$$\zer(\partial f+\nabla g)=\argmin_{x\in {\cal H}}\{f(x)+g(x)\}$$
and the fact that $\nabla g$ is $\beta$-cocoercive due to the Baillon-Haddad Theorem (see \cite[Corollary 18.16]{bauschke-book}). 
\end{proof}

\section{A Douglas-Rachford algorithm with Tikhonov regularization term}\label{sec4}

In this section we derive from the Krasnosel'ski\u{\i}--Mann algorithm formulated in Section \ref{sec2} an iterative scheme of Douglas-Rachford-type, which generates sequences 
that strongly converge to a zero of of the sum of two set-valued maximally monotone operators. 

In what follows, we denote by $R_{A}=2J_{A}-\id$ the reflected resolvent of a maximally monotone operator $A:{\cal H}\rightrightarrows {\cal H}$. 

\begin{theorem}\label{tikhonov-dr-operators} Let $A,B:{\cal H}\rightrightarrows{\cal H}$ be two
maximally monotone operators such that $\zer(A+B)\neq\emptyset$ and $\gamma >0$. Consider the following iterative scheme: 
$$(\forall n \geq 0)\hspace{0.2cm}\left\{
\begin{array}{ll}
y_n=J_{\gamma B}(\beta_nx_n)\\
z_n=J_{\gamma A}(2y_n-\beta_nx_n)\\
x_{n+1}=\beta_nx_n+\lambda_n(z_n-y_n)
\end{array}\right.$$
with $x_0\in{\cal H}$ as starting point and $(\lambda_n)_{n\geq 0}$ and $(\beta_n)_{n\geq 0}$ real sequences satisfying  the conditions: 

(i) $0<\beta_n\leq 1$ for any $n\geq 0$, $\lim_{n\rightarrow+\infty}\beta_n=1$, $\sum_{n\geq 0} (1-\beta_n)=+\infty$ and 
$\sum_{n \geq 1} |\beta_n-\beta_{n-1}|<+\infty$; 

(ii) $0<\lambda_n\leq 2$ for any $n\geq 0$, $\liminf_{n\rightarrow+\infty}\lambda_n>0$ and 
$\sum_{n\geq 1} |\lambda_n-\lambda_{n-1}|<+\infty$.

Then the following statements are true:
\begin{itemize} \item[(a)] $(x_n)_{n\geq 0}$ converges strongly to $\ol x:=\proj_{\fix R_{\gamma A}R_{\gamma B}}(0)$ as $n\rightarrow+\infty$; 
\item[(b)] $(y_n)_{n\geq 0}$ and $(z_n)_{n\geq 0}$ converge strongly to $J_{\gamma B}(\ol x)\in\zer(A+B)$ as $n\rightarrow+\infty$.
\end{itemize} 
\end{theorem}

\begin{proof} Taking into account the iteration rules and the definition of the reflected resolvent, the iterative scheme in the  enunciation of the theorem can be equivalently written as
\begin{align}\label{DR-kr-m} x_{n+1}= & \  \beta_nx_n+\lambda_n\Big[J_{\gamma A}\circ(2J_{\gamma B}-\id)(\beta_nx_n)-J_{\gamma B}(\beta_nx_n)\Big]\nonumber\\
                    = & \ \beta_nx_n+\lambda_n\left[\left(\frac{\id +R_{\gamma A}}{2}\circ R_{\gamma B}\right)(\beta_nx_n)-\frac{\id +R_{\gamma B}}{2}(\beta_nx_n)\right]\nonumber\\
                    = & \ \beta_nx_n+\frac{\lambda_n}{2}\big(T(\beta_nx_n)-\beta_nx_n\big) \ \forall n \geq 0,
\end{align}
where $T:=R_{\gamma A}\circ R_{\gamma B} : {\cal H} \rightarrow {\cal H}$ is a nonexpansive operator (see 
\cite[Corollary 23.10(ii)]{bauschke-book}). 
From \cite[Proposition 25.1(ii)]{bauschke-book} we have $\zer(A+B)=J_{\gamma B}(\fix T)$, hence  
$\fix T\neq\emptyset$. By applying Theorem \ref{tikhonov-nonexp}, we obtain that $(x_n)_{n\geq 0}$ converges strongly 
to $\ol x:=\proj_{\fix T}(0)$ as $n\rightarrow+\infty$, hence (a) holds. 

Further, by taking into account the definition of the sequence $(y_n)_{n\geq 0}$ and the continuity of the resolvent operator, we obtain
that $(y_n)_{n\geq 0}$ converges strongly to $J_{\gamma B}\ol x\in\zer(A+B)$ as $n\rightarrow+\infty$. Finally, by taking the limit 
in the recursive formula 
of the sequence $(x_n)_{n\geq 0}$, we obtain that $z_n-y_n$ converges strongly to $0$ as $n\rightarrow+\infty$, thus (b) holds, too. 
\end{proof}

\begin{remark}\label{dr} \rm The classical Douglas-Rachford method, which reads
$$(\forall k\geq 0)\hspace{0.2cm}\left\{
\begin{array}{ll}
y_n=J_{\gamma B}(x_n)\\
z_n=J_{\gamma A}(2y_n-x_n)\\
x_{n+1}=x_n+\lambda_n(z_n-y_n),
\end{array}\right.$$
produces sequences for which in general only weak convergence to a zero of the $A+B$ can be proved (see for example \cite[Theorem 25.6]{bauschke-book}). 
In order to ensure strong convergence, one usually needs to impose restrictve conditions on the monotone operators involved, like uniform monotonicity (which is a generalization of strong monotonicity). This is not the case for the iterative scheme stated in Theorem \ref{tikhonov-dr-operators}, where we are able to guarantee strong convergence in the very general situation of maximally monotone operators.  
\end{remark}

Further, we look at  optimization problems of the form
\begin{equation}\label{pr-opt-dr}
\min_{x \in {\cal H}}\{ f(x) + g(x)\},
\end{equation}
where $f,g:{\cal H}\rightarrow\R\cup\{+\infty\}$ are proper, convex and lower semicontinuous functions. 
We denote by 
$$\dom f =\{x\in{\cal H}: f(x)<+\infty \}$$
the effective domain of the function $f$. 

In order to proceed, we need the following notion. For $S\subseteq {\cal H}$ a convex set, we denote by
$$\sqri S:=\{x\in S:\cup_{\lambda>0}\lambda(S-x) \ \mbox{is a closed linear subspace of} \ {\cal H}\}$$
its strong quasi-relative interior. Notice that we always have $\inte S\subseteq\sqri S$ (in general this inclusion may be strict). 
If ${\cal H}$ is finite-dimensional, then $\sqri S$ coincides with $\ri S$, the relative interior of $S$, which is the interior of $S$ 
with respect to its affine hull. The notion of strong quasi-relative interior belongs to the class of generalized interiority notions and plays 
an important role in the formulation of regularity conditions, which are needed in convex optimization
in order to guarantee duality results and also subdifferential sum formulas. The one considered in the next result is the so-called Attouch-Br\'{e}zis regularity condition. We refer to \cite{bo-van, bauschke-book, simons, Zal-carte} for more interiority notions and their impact on the duality theory. 

\begin{corollary}\label{tikhonov-dr-optimization} Let $f,g:{\cal H}\rightarrow\R\cup\{+\infty\}$ be proper, convex and lower semicontinuous 
functions such that $\argmin_{x\in {\cal H}}\{f(x)+g(x)\}\neq\emptyset$ and $0\in\sqri(\dom f-\dom g)$ and $\gamma >0$. Consider the following iterative 
scheme: 
$$(\forall n\geq 0)\hspace{0.2cm}\left\{
\begin{array}{ll}
y_n=\prox_{\gamma g}(\beta_nx_n)\\
z_n=\prox_{\gamma f}(2y_n-\beta_nx_n)\\
x_{n+1}=\beta_nx_n+\lambda_n(z_n-y_n)
\end{array}\right.$$
with $x_0\in{\cal H}$ as starting point and $(\lambda_n)_{n\geq 0}$ and $(\beta_n)_{n\geq 0}$ real sequences satisfying  the conditions: 

(i) $0<\beta_n\leq 1$ for any $n\geq 0$, $\lim_{n\rightarrow+\infty}\beta_n=1$, $\sum_{n\geq 0} (1-\beta_n)=+\infty$ and 
$\sum_{n \geq 1} |\beta_n-\beta_{n-1}|<+\infty$; 

(ii) $0<\lambda_n\leq 2$ for any $n\geq 0$, $\liminf_{n\rightarrow+\infty}\lambda_n>0$ and 
$\sum_{n\geq 1} |\lambda_n-\lambda_{n-1}|<+\infty$.

Then the following statements are true:
\begin{itemize} \item[(a)] $(x_n)_{n\geq 0}$ converges strongly to $\ol x:=\proj_{\fix T}(0)$ as $n\rightarrow+\infty$, where 
$T=(2\prox_{\gamma f}-\id)\circ(2\prox_{\gamma g}-\id)$; 
\item[(b)] $(y_n)_{n\geq 0}$ and $(z_n)_{n\geq 0}$ converge strongly to $\prox_{\gamma g}(\ol x)\in\argmin_{x\in {\cal H}}\{f(x)+g(x)\}$ 
as $n\rightarrow+\infty$.
\end{itemize} 
\end{corollary}

\begin{proof} The results is a direct consequence of Theorem \ref{tikhonov-dr-operators} for $A=\partial f$ and $B=\partial g$ and by noticing 
that the regularity condition $0\in\sqri(\dom f-\dom g)$ ensures the relation (see \cite[Proposition 7.2]{bauschke-book})
$$\zer(\partial f+\partial g)=\argmin_{x\in {\cal H}}\{f(x)+g(x)\}.$$
\end{proof}

\section{Strongly convergent primal-dual algorithms}\label{sec5}

The aim of this section is to induce strong convergence in the nowadays so popular \emph{primal-dual algorithms} designed for solving
highly structured monotone inclusions involving parallel sums and compositions with linear operators. 

\subsection{A primal-dual algorithm of forward-backward-type with Tikhonov regularization terms}\label{sec51}

In this subsection, the following monotone inclusion problem will be in the focus of our investigations.

\begin{problem}\label{pr1}
Let $A:{\cal H}\rightrightarrows {\cal H}$ be a maximally monotone operator and $C:{\cal H}\rightarrow {\cal H}$ a $\mu$-cocoercive operator, for some $\mu>0$. Let $m$ be a strictly positive integer and for any 
$i = 1,...,m,$ let ${\cal G}_i$  
be a real Hilbert space, $B_i,\,D_i:{\cal G }_i \rightrightarrows {\cal G}_i$ be maximally monotone operators such that 
$D_i$ are $\nu_i$-strongly monotone, for some $\nu_i>0$, and 
$L_i:{\cal H}\rightarrow$ ${\cal G}_i$ be a nonzero linear continuous operator. The problem is to solve the primal inclusion
\begin{equation}\label{sum-k-primal-C-D}
\mbox{find } \ol x \in {\cal H} \ \mbox{such that} \ 0\in A\ol x+ \sum_{i=1}^{m}L_i^*(B_i\Box D_i)(L_i \ol x) + C \ol x,
\end{equation}
together with the dual inclusion of Attouch-Th\'{e}ra type (see \cite{AttThe96, combettes-pesquet, vu})
\begin{equation}\label{sum-k-dual-C-D}
\mbox{ find } \ol v_1 \in {\cal G}_1,...,\ol v_m \in {\cal G}_m \ \mbox{such that } \exists x\in {\cal H}: \
\left\{
\begin{array}{ll}
-\sum_{i=1}^{m}L_i^*\ol v_i\in Ax+Cx\\
\ol v_i\in (B_i\Box D_i)(L_ix),\ i=1,...,m.
\end{array}\right.
\end{equation}
\end{problem}

Some of the notations used above are to be specified.  The operator $L_i^*: {\cal G }_i \rightarrow {\cal H}$, defined via 
$\< L_ix,y  \> = \< x,L_i^*y  \>$ for all $x \in {\cal H}$ and all $y \in {\cal G }_i$, denotes the adjoint of the linear continuous operator $L_i: {\cal H} \rightarrow {\cal G }_i$,
for $i = 1,\ldots,m$. We say that $D_i:{\cal G}_i \rightrightarrows {\cal G}_i$ is $\nu_i$-strongly monotone, for some $\nu_i>0$, if $\langle x-y,u-v\rangle\geq \nu_i\|x-y\|^2$ for all $(x,u),(y,v)\in\gr D_i, i=1,...,m.$ 
The parallel sum of the set-valued operators $B_i,\,D_i:{\cal G}_i \rightrightarrows {\cal G}_i$ is defined as
$B_i \Box D_i:{\cal G}_i \rightrightarrows {\cal G}_i,\, B_i \Box D_i  = \left(B_i^{-1} + D_i^{-1}\right)^{-1}$, for $i=1,...,m$.

We say that $(\ol x, \ol v_1,...,\ol v_m)\in{\cal H} \times$ ${\cal G}_1 \times...\times {\cal G}_m$ is a primal-dual solution to Problem 
\ref{pr1}, if
\begin{equation}\label{prim-dual-C-D}-\sum_{i=1}^{m}L_i^*\ol v_i\in A\ol x+C\ol x \mbox{ and }\ol v_i\in (B_i\Box D_i)(L_i\ol x),\ i=1,...,m.\end{equation}

It is easy to see that, if $(\ol x, \ol v_1,...,\ol v_m)\in{\cal H} \times$ ${\cal G}_1 \times...\times {\cal G}_m$ is a primal-dual solution to Problem \ref{pr1}, 
then $\ol x$ is a solution to \eqref{sum-k-primal-C-D} and $(\ol v_1,...,\ol v_m)\in$ ${\cal G}_1 \times...\times {\cal G}_m$ is a solution to 
\eqref{sum-k-dual-C-D}. Moreover, if $\ol x\in{\cal H}$ is a solution to \eqref{sum-k-primal-C-D}, then there exists $(\ol v_1,...,\ol v_m)\in$ ${\cal G}_1 \times...\times {\cal G}_m$ 
such that $(\ol x, \ol v_1,...,\ol v_m)$ is a primal-dual solution to Problem \ref{pr1} and, if 
$(\ol v_1,...,\ol v_m)\in$ ${\cal G}_1 \times...\times {\cal G}_m$ is a solution to \eqref{sum-k-dual-C-D}, then there exists 
$\ol x\in{\cal H}$ such that $(\ol x, \ol v_1,...,\ol v_m)$ is a primal-dual solution to Problem \ref{pr1}. 

\begin{theorem}\label{fb-pd-tikhonov} In Problem \ref{pr1}, suppose that
\begin{align}\label{0-range-fb}
	0\in\ran \left(A+\sum_{i=1}^m L_i^*\circ (B_i\Box D_i)\circ L_i+C\right).
\end{align}

\noindent Let $\tau$ and $\sigma_i$, $i=1,...,m$, be strictly positive numbers such that 
$$2\cdot\min\{\tau^{-1},\sigma_1^{-1},...,\sigma_m^{-1}\}\cdot\min\{\mu,\nu_1,...,\nu_m\}\left(1-\sqrt{\tau\sum_{i=1}^{m}\sigma_i\|L_i\|^2}\right)\geq 1.$$
Let be the starting point $(x_0,v_{1,0},...,v_{m,0}) \in {\cal H} \times $ ${\cal G}_1$ $\times...\times$ ${\cal G}_m$ and set:

\begin{align}\label{fb_pd-alg}
	  \left(\forall n\geq 0\right) \begin{array}{l} \left\lfloor \begin{array}{l}
		p_n = J_{\tau A}\big[\beta_n x_n-\tau\big(\beta_n\sum_{i=1}^{m}L_i^*v_{i,n}+C(\beta_n x_n)\big)\big] \\
		x_{n+1} = \beta_n x_n+\lambda_n(p_n-\beta_nx_n)\\
		\text{For }i=1,\!...,m  \\
				\left\lfloor \begin{array}{l}
					q_{i,n} =  J_{\sigma_i B_i^{-1}}\big[\beta_nv_{i,n}+\sigma_i\big(L_i(2p_n-\beta_nx_n)-D_i^{-1}(\beta_nv_{i,n})\big)\big] \\
					v_{i,n+1} = \beta_nv_{i,n}+\lambda_n(q_{i,n}-\beta_nv_{i,n})
				\end{array} \right.\\
		 \vspace{-4mm}
		\end{array}
		\right.
		\end{array}
\end{align}

where $(\lambda_n)_{n\geq 0}$ and $(\beta_n)_{n\geq 0}$ are real sequences satisfying the conditions: 

(i) $0<\beta_n\leq 1$ for any $n\geq 0$, $\lim_{n\rightarrow+\infty}\beta_n=1$, $\sum_{n\geq 0} (1-\beta_n)=+\infty$ and 
$\sum_{n \geq 1} |\beta_n-\beta_{n-1}|<+\infty$; 

(ii) $0<\lambda_n\leq \frac{4\beta\rho-1}{2\beta\rho}$ for any $n\geq 0$, $\liminf_{n\rightarrow+\infty}\lambda_n>0$ and 
$\sum_{n\geq 1} |\lambda_n-\lambda_{n-1}|<+\infty$, 

for $$\beta=\min\{\mu,\nu_1,...,\nu_m\}$$ and 
$$\rho=\min\{\tau^{-1},\sigma_1^{-1},...,\sigma_m^{-1}\}\left(1-\sqrt{\tau\sum_{i=1}^{m}\sigma_i\|L_i\|^2}\right).$$

Then there exists a primal-dual solution $(\ol x, \ol v_1,..., \ol v_m)$ to Problem \ref{pr1} such that the sequence of primal-dual iterates
$(x_n,v_{1,n},...,v_{m,n})$ converges strongly to $(\ol x, \ol v_1,...,\ol v_m)$ as $n \rightarrow +\infty$.
\end{theorem}

\begin{remark} \rm (i) Since $D_i:{\cal G }_i \rightrightarrows {\cal G}_i$ is $\nu_i$-strongly monotone, we have that 
$D_i^{-1}:{\cal G }_i \rightarrow {\cal G}_i $ is $\nu_i$-cocoercive, for $i=1,...,m$. 

(ii) The resolvent of the inverse operator of a maximally monotone operator $M:{\cal H}\rightrightarrows {\cal H}$ can be computed as follows (see \cite{bauschke-book}):
\begin{align}
	\label{res-identity}
	\id = J_{\gamma M} + \gamma J_{\gamma^{-1}M^{-1}}\circ \gamma^{-1} \id.
\end{align}
\end{remark}

\begin{proof} The idea is to apply Theorem \ref{tikhonov-fb-operators} in an appropriate product space under the use of appropriate renorming techniques (see \cite{vu}).
We consider the Hilbert space $\fK = \h \times \g_1 \times \,... \times \g_m$ endowed with inner product and associated norm defined, for  $(x,v_1,\!...,v_m)$, $(y,q_1,\!...,q_m) \in \fK$, via
\begin{align}
	\label{dr_def1.001}
	\begin{aligned}
	\< (x,v_1,\!...,v_m),(y,q_1,\!...,q_m) \>_{\fK} &= \<x,y\>_{\h} + \sum_{i=1}^m \< v_i,q_i \>_{\g_i} \\ \text{ and } \|(x,v_1,\!...,v_m)\|_{\fK} &= \sqrt{\|x\|_{\h}^2 + \sum_{i=1}^m \| v_i \|_{\g_i}^2},
	\end{aligned}
\end{align}
respectively. Furthermore, we consider the set-valued operator
\begin{align*}
	\f M : \fK \rightrightarrows \fK, \quad (x,v_1,\!...,v_m) \mapsto ( Ax,  B_1^{-1}v_1, \!...,  B_m^{-1}v_m),
\end{align*}
which is maximally monotone,  since $A$ and $B_i$, $i=1,\!...,m,$ are maximally monotone (see \cite[Proposition 20.22 and Proposition 20.23]{bauschke-book}), and the linear continuous operator
\begin{align*}
	\f S : \fK \rightarrow \fK, \quad (x,v_1,\!...,v_m) \mapsto \left(\sum_{i=1}^m L_i^* v_i, -L_1 x , \!..., - L_m x\right),
\end{align*}
which is skew-symmetric (i.\,e. $\f S^*=-\f S$) and hence maximally monotone (see \cite[Example 20.30]{bauschke-book}). 
We also consider the single-valued operator
\begin{align*}
	\f Q : \fK \rightarrow \fK, \quad (x,v_1,\!...,v_m) \mapsto \left(Cx, D_1^{-1} v_1, \!..., D_m^{-1} v_m\right),
\end{align*}
which is once again maximally monotone, since $C$ and $D_i$ are maximally monotone for $i=1,\!...,m$. 
Therefore, since $\dom \f S = \fK$, $\f M+\f S$ is maximally monotone (see \cite[Corollary 24.4(i)]{bauschke-book}). 
According to \cite[page 672]{vu} $$\f Q \mbox{ is }\beta-\mbox{cocoercive}.$$
Further, one can easily verify that \eqref{0-range-fb} is equivalent to $\zer\left(\f M + \f S + \f Q\right) \neq \varnothing$ and 
(see also \cite[page 317]{combettes-pesquet})
\begin{align}
	\label{optcond}
	\begin{aligned}
	&(x,v_1,\!...,v_m) \in \zer\left(\f M + \f S + \f Q\right) \\ 
	\Leftrightarrow &(x, v_1,\!...,v_m) \text{ is a primal-dual solution to Problem } \ref{pr1}.
	\end{aligned}
\end{align}
We also introduce the linear continuous operator
\begin{align*}
	\f V : \fK \rightarrow \fK, \quad (x,v_1,\!...,v_m) \mapsto \left(\frac{x}{\tau} -\sum_{i=1}^m L_i^* v_i, \frac{v_1}{\sigma_1} -  L_1x , \!..., \frac{v_m}{\sigma_m} - L_m x\right),
\end{align*}
which is self-adjoint and $\rho$-strongly positive (see \cite{vu}), 
namely, the following inequality holds $$\< \f x,\f V\f x\>_{\fK}\geq \rho\|\f x\|_{\fK}^2 \ \forall \f x\in\fK.$$
Therefore, its inverse operator $\f V^{-1}$ exists and it fulfills $\|\f V^{-1}\|\leq \frac{1}{\rho}$.

The algorithmic scheme \eqref{fb_pd-alg} in the statement of the theorem can be written by using this notations as

\begin{align}\label{fb_pd-alg-1}
	  \left(\forall n\geq 0\right) \!\! \begin{array}{l} \left\lfloor \begin{array}{l}
		\beta_n(\tau^{-1}x_n-\sum_{i=1}^m L_i^*v_{i,n})-\tau^{-1}p_n+\sum_{i=1}^m L_i^*q_{i,n}-C(\beta_nx_n)\in Ap_n+\sum_{i=1}^m L_i^*q_{i,n} \\
		x_{n+1} = \beta_n x_n+\lambda_n(p_n-\beta_nx_n)\\
		\text{For }i=1,\!...,m  \\
				\left\lfloor \begin{array}{l}
					\beta_n(\sigma_i^{-1}v_{i,n}-L_ix_n)-\sigma_i^{-1}q_{i,n}+L_ip_n-D_i^{-1}(\beta_nv_{i,n})\in B_i^{-1}(q_{i,n})-L_ip_n \\
					v_{i,n+1} = \beta_nv_{i,n}+\lambda_n(q_{i,n}-\beta_nv_{i,n}),
				\end{array} \right.\\
		 \vspace{-4mm}
		\end{array}
		\right.
		\end{array}
\end{align}
By introducing the sequences 
\begin{align*}
 \fx_n = (x_n,v_{1,n},\!...,v_{m,n}),\ \mbox{and} \ \fy_n =(p_n,q_{1,n},\!...,q_{m,n}) \ \forall n \geq 0,
\end{align*}
the scheme \eqref{fb_pd-alg-1} can equivalently be written in the form
\begin{align}
	\label{fb_pd-alg-2}
	\left(\forall n\geq 0\right)  \left\lfloor \begin{array}{l}
	\beta_n \f V(\fx_n) - \f V(\fy_n) -\f Q(\beta_n\fx_n) \in \left(\f M+\f S\right)(\fy_n) \\
	\fx_{n+1} =\beta_n \fx_n + \lambda_n \left(\fy_n-\beta_n\fx_n\right).
	\end{array}
	\right.
\end{align}
Furthermore, we have for any $n \geq 0$
\begin{align*}
& \ \beta_n \f V(\fx_n) - \f V(\fy_n) -\f Q(\beta_n\fx_n) \in \left(\f M+\f S\right)(\fy_n)\\
\Leftrightarrow & \ \big(\beta_n\f V-\f Q\circ(\beta_n\id)\big)(\fx_n)\in (\f M+\f S+\f V)(\f y_n)\\
\Leftrightarrow & \ \f y_n=(\f M+\f S+\f V)^{-1} \big(\beta_n\f V-\f Q\circ(\beta_n\id)\big)(\f x_n)\\
\Leftrightarrow & \ \f y_n=\big(\id+\f V^{-1}(\f M+\f S)\big)^{-1} \big(\beta_n\id-\f V^{-1}\circ\f Q\circ(\beta_n\id)\big)(\f x_n)\\
\Leftrightarrow & \ \f y_n=J_{\f A}\big(\beta_n\f x_n-\f B(\beta_n\f x_n)\big),
\end{align*}
where \begin{align}\label{a-b}
		\f A:=\f V^{-1}\left(\f M+\f S \right) \ \mbox{and} \ \f B := \f V^{-1}\f Q. 
\end{align}

Let $\fK_{\f V}$ be the Hilbert space with inner product and norm defined, for $\fx,\fy \in \fK$, by
\begin{align}\label{dr_HSKV}
 \< \fx,\fy \>_{\fK_{\f V}} = \< \fx, \f V \fy \>_{\fK} \text{ and } \|\fx\|_{\fK_{\f V}} = \sqrt{\< \fx, \f V \fx \>_{\fK}},
\end{align}
respectively. As the set-valued operators $\f M+\f S$ and $ \f Q$ are maximally monotone on $\fK$, the operators
$\f A$ and $\f B$ 
are maximally monotone on $\fK_{\f V}$ (see also \cite{vu}). Furthermore, $\f B$ is $\beta\rho$-cocoercive on $\fK_{\f V}$. Moreover, since $\f V$ is self-adjoint and $\rho$-strongly positive, 
weak and strong convergence in $\fK_{\f V}$ are equivalent with weak and strong convergence in $\fK$, respectively.

Taking this into account, it follows that \eqref{fb_pd-alg-2} becomes
$$\left(\forall n\geq 0\right) \ \fx_{n+1} =\beta_n \fx_n + \lambda_n \Big(J_{\f A}\big(\beta_n\f x_n-\f B(\beta_n\f x_n)\big)-\beta_n\fx_n\Big),$$
which is the algorithm presented in Theorem \ref{tikhonov-fb-operators} for determining the zeros of $\f A+\f B$ in case $\gamma=1$. However, we have
$$\zer(\f A+\f B) = \zer(\f V^{-1}\left(\f M+\f S + \f Q\right)) = \zer(\f M + \f S + \f Q).$$
According to Theorem \ref{tikhonov-fb-operators}, $\f x_n$ converges strongly to $\proj_{\zer(\f A+\f B)}(0,0,...,0)$ in the space 
$\fK_{\f V}$ as $n\rightarrow+\infty$ and the conclusion follows from \eqref{optcond}. 
\end{proof}

In the remaining of this subsection we investigate the convergence property of the algorithm \eqref{fb_pd-alg}
in the context of simultaneously solving complexly structured convex optimization problems and their Fenchel duals. The problem under investigation is the following one.

\begin{problem}\label{pr} Let $f\in\Gamma({\cal H})$ and $h:{\cal H}\rightarrow \R$ be a convex and differentiable function with a $\mu^{-1}$-Lipschitz continuous gradient, for some $\mu>0$.
Let $m$ be a strictly positive integer and for $i=1, ..., m$, let ${\cal G}_i$  be a real Hilbert space,
$g_i, l_i \in\Gamma({\cal G}_i)$ such that $l_i$ is $\nu_i$-strongly convex, for some $\nu_i > 0$
and $L_i:{\cal H}\rightarrow$ ${\cal G}_i$ a nonzero linear continuous operator. Consider the convex optimization problem
\begin{equation}\label{sum-k-prim-f2-opt}
\inf_{x\in {\cal H}}\left\{f(x)+\sum_{i=1}^{m}(g_i \Box l_i)(L_ix)+h(x)\right\}
\end{equation}
and its Fenchel-type dual problem
\begin{equation}\label{sum-k-dual-f2-opt}
\sup_{v_i\in {\cal{G}}_i,\, i=1,\!...,m}\left\{-\big(f^*\Box h^*\big)\left(-\sum_{i=1}^{m}L_i^*v_i\right)-
\sum_{i=1}^{m}\big(g_i^*(v_i)+l_i^*(v_i) \big) \right\}.
\end{equation}
\end{problem}
We denote by $\Gamma({\cal H})$ the set of proper, convex and lower semicontinuous functions defined on ${\cal H}$ with values 
in the extended real line $\R\cup\{+\infty\}$. The conjugate of a function $f$ is $f^*:\h \rightarrow \B$, 
$f^*(p)=\sup{\left\{ \left\langle p,x \right\rangle -f(x) : x\in\h \right\}}$ for all $p \in \h$. Moreover, if $f \in \Gamma(\h)$, 
then $f^* \in \Gamma(\h)$, as well, and $(\partial f)^{-1}=\partial f^*$. Finally, having two proper functions $f,\,g : \h \rightarrow \B$, their infimal convolution
is defined by $f \Box g : \h \rightarrow \B$, $(f \Box g) (x) = \inf_{y \in \h}\left\{ f(y) + g(x-y) \right\}$ for all $x \in \h$. 

\begin{corollary}\label{fb-pd-tikhonov-opt} In Problem \ref{pr}, suppose that
\begin{align}\label{0-range-fb-opt}
	0\in\ran \left(\partial f+\sum_{i=1}^m L_i^*\circ (\partial g_i\Box \partial l_i)\circ L_i+\nabla h\right).
\end{align}

\noindent Let $\tau$ and $\sigma_i$, $i=1,...,m$, be strictly positive numbers such that 
$$2\cdot\min\{\tau^{-1},\sigma_1^{-1},...,\sigma_m^{-1}\}\cdot\min\{\mu,\nu_1,...,\nu_m\}\left(1-\sqrt{\tau\sum_{i=1}^{m}\sigma_i\|L_i\|^2}\right)\geq 1.$$
Let be the starting point $(x_0,v_{1,0},...,v_{m,0}) \in {\cal H} \times $ ${\cal G}_1$ $\times...\times$ ${\cal G}_m$ and set:

\begin{align}\label{fb_pd-alg-opt}
	  \left(\forall n\geq 0\right) \begin{array}{l} \left\lfloor \begin{array}{l}
		p_n = \prox_{\tau f}\big[\beta_n x_n-\tau\big(\beta_n\sum_{i=1}^{m}L_i^*v_{i,n}+\nabla h(\beta_n x_n)\big)\big] \\
		x_{n+1} = \beta_n x_n+\lambda_n(p_n-\beta_nx_n)\\
		\text{For }i=1,\!...,m  \\
				\left\lfloor \begin{array}{l}
					q_{i,n} =  \prox_{\sigma_i g_i^*}\big[\beta_nv_{i,n}+\sigma_i\big(L_i(2p_n-\beta_nx_n)-\nabla l_i^*(\beta_nv_{i,n})\big)\big] \\
					v_{i,n+1} = \beta_nv_{i,n}+\lambda_n(q_{i,n}-\beta_nv_{i,n}),
				\end{array} \right.\\
		 \vspace{-4mm}
		\end{array}
		\right.
		\end{array}
\end{align}

where $(\lambda_n)_{n\geq 0}$ and $(\beta_n)_{n\geq 0}$ are real sequences satisfying the conditions: 

(i) $0<\beta_n\leq 1$ for any $n\geq 0$, $\lim_{n\rightarrow+\infty}\beta_n=1$, $\sum_{n\geq 0} (1-\beta_n)=+\infty$ and 
$\sum_{n \geq 1} |\beta_n-\beta_{n-1}|<+\infty$; 

(ii) $0<\lambda_n\leq \frac{4\beta\rho-1}{2\beta\rho}$ for any $n\geq 0$, $\liminf_{n\rightarrow+\infty}\lambda_n>0$ and 
$\sum_{n\geq 1} |\lambda_n-\lambda_{n-1}|<+\infty$, 

for $$\beta:=\min\{\mu,\nu_1,...,\nu_m\}$$ and 
$$\rho:=\min\{\tau^{-1},\sigma_1^{-1},...,\sigma_m^{-1}\}\left(1-\sqrt{\tau\sum_{i=1}^{m}\sigma_i\|L_i\|^2}\right).$$

Then there exists $(\ol x, \ol v_1,..., \ol v_m)\in \h \times \g_1 \times \,... \times \g_m$ such that 
$(x_n,v_{1,n},...,v_{m,n})$ converges strongly to $(\ol x, \ol v_1,...,\ol v_m)$ as $n \rightarrow +\infty$ and
$\ol x$ is an optimal solution of the problem \eqref{sum-k-prim-f2-opt}, $(\ol v_1,\!...,\ol v_m)$ is an optimal solution of 
\eqref{sum-k-dual-f2-opt} and the optimal objective values of the two optimization problems coincide.
\end{corollary}

\begin{remark} \rm The proximal-point operator of the conjugate function can 
be computed via the Moreau's decomposition formula
\begin{equation}\label{prox-f-star}
\prox\nolimits_{\gamma f}+\gamma\prox\nolimits_{(1/\gamma)f^*}\circ\gamma^{-1}\id=\id,
\end{equation}
which is valid for $\gamma>0$ and $f\in \Gamma({\cal H})$ (see \cite{bauschke-book}).  
\end{remark}

\begin{proof} Consider the maximal monotone operators
$$A=\partial f, C=\nabla h, B_i=\partial g_i \ \mbox{and} \ D_i=\partial l_i,\ i=1,\!...,m.$$
The Baillon-Haddad Theorem (see \cite[Corollary 18.16]{bauschke-book}) ensures that $C$ is $\mu$-cocoercive. 
Since $l_i$ is $\nu_i$-strongly convex, $D_i$ is $\nu_i$-strongly monotone, for $i=1,...,m.$ According to 
\cite[Proposition 17.10, Theorem 18.15]{bauschke-book}, $D_i^{-1} = \nabla l_i^*$ is a monotone and $\nu_i^{-1}$-Lipschitz  
continuous operator for $i=1,\!...,m$. The strong convexity of the functions $l_i$ guarantees that $g_i \Box l_i \in \Gamma({\cal G}_i)$ 
(see \cite[Corollary 11.16, Proposition 12.14]{bauschke-book}) and 
$\partial (g_i \Box l_i) = \partial g_i \Box \partial l_i, i=1,...,m$, (see \cite[Proposition 24.27]{bauschke-book}).

Hence, the monotone inclusion problem \eqref{sum-k-primal-C-D} reads
\begin{equation}\label{sum-k-primal-C-D-f}
\mbox{find } \ol x \in {\cal H} \ \mbox{such that } 0\in \partial f(\ol x)+ \sum_{i=1}^{m}L_i^*(\partial g_i \Box \partial l_i) (L_i \ol x)+ \nabla h(\ol x),
\end{equation}
while the dual monotone inclusion problem \eqref{sum-k-dual-C-D} reads
\begin{equation}\label{sum-k-dual-C-D-f}
\mbox{ find } \ol v_1 \in {\cal G}_1,\!...,\ol v_m \in {\cal G}_m \ \mbox{such that } \exists x\in {\cal H}: \
\left\{
\begin{array}{ll}
-\sum_{i=1}^{m}L_i^*\ol v_i\in \partial f(x)+\nabla h(x)\\
\ol v_i\in (\partial g_i \Box \partial l_i) (L_ix),\ i=1,\!...,m.
\end{array}\right.
\end{equation}

One can see that if $(\ol x, \ol v_1,\!...,\ol v_m)\in{\cal H} \times$ ${\cal{G}}_1 \times...\times {\cal{G}}_m$ is a primal-dual solution to 
\eqref{sum-k-primal-C-D-f}-\eqref{sum-k-dual-C-D-f}, namely,
\begin{equation}\label{prim-dual-f2}-\sum_{i=1}^{m}L_i^*\ol v_i\in \partial f(\ol x)+\nabla h(\ol x) \mbox{ and }\ol v_i\in (\partial g_i \Box \partial l_i)(L_i\ol x),\ i=1,\!...,m,\end{equation}
then $\ol x$ is an optimal solution of the problem \eqref{sum-k-prim-f2-opt}, $(\ol v_1,\!...,\ol v_m)$ is an optimal solution of 
\eqref{sum-k-dual-f2-opt} and the optimal objective values of the two problems coincide. Notice that \eqref{prim-dual-f2} is nothing 
else than the system of optimality conditions for the primal-dual pair of convex optimization problems 
\eqref{sum-k-prim-f2-opt}-\eqref{sum-k-dual-f2-opt}.

The conclusion follows now from Theorem \ref{fb-pd-tikhonov}. 
\end{proof}

\begin{remark}\label{existence-opt-sol-reg-cond} \rm (i) The relation \eqref{0-range-fb-opt} in the above theorem is fulfilled if the primal 
problem \eqref{sum-k-prim-f2-opt} has an optimal solution $\ol x\in\h$ and a suitable regularity condition holds. Under these auspices there exists an optimal solution to 
\eqref{sum-k-dual-f2-opt} $(\ol v_1,\!...,\ol v_m)\in {\cal{G}}_1 \times...\times {\cal{G}}_m$, 
such that $(\ol x, \ol v_1,\!...,\ol v_m)$ satisfies the optimality conditions \eqref{prim-dual-f2} and, consequently, \eqref{0-range-fb-opt} holds.

(ii) Further, let us discuss some conditions ensuring the existence of a primal optimal solution.
Suppose that the primal problem \eqref{sum-k-prim-f2-opt} is feasible, which means that its optimal objective
value is not identical $+\infty$. The existence of optimal solutions for \eqref{sum-k-prim-f2-opt} is guaranteed if, 
for instance, $f+h$ is coercive (that is $\lim_{\|x\|\rightarrow\infty}(f+h)(x)=+\infty$)
and for all $i=1,...,m$, $g_i$ is bounded from below. Indeed, under these circumstances, the objective function of
\eqref{sum-k-prim-f2-opt} is coercive (use also \cite[Corollary 11.16 and Proposition 12.14]{bauschke-book} to show that
for all $i=1,...,m$, $g_i\Box l_i$ is bounded from below and $g_i\Box l_i\in\Gamma({\cal{G}}_i)$) and the statement follows
via \cite[Corollary 11.15]{bauschke-book}. On the other hand, if $f+h$ is strongly convex, then the objective function of
\eqref{sum-k-prim-f2-opt} is strongly convex, too, thus \eqref{sum-k-prim-f2-opt} has a unique optimal solution
(see \cite[Corollary 11.16]{bauschke-book}).

(iii) We discuss at this point a suitable regularity condition as mentioned at item (i) above. Since $\dom (g_i \Box l_i) = \dom g_i + \dom l_i,\ i=1,\!...,m$, 
one can use to this end the regularity condition of interiority-type  (see also \cite{combettes-pesquet})
\begin{equation}\label{reg-cond}
(0,\!...,0)\in\sqri\left(\prod_{i=1}^{m}(\dom g_i + \dom l_i)-\{(L_1x,\!...,L_mx):x\in \dom f\}\right).
\end{equation}
This is fulfilled provided that one of the following conditions is verified (see \cite[Proposition 4.3]{combettes-pesquet}):
\begin{enumerate}
\item [(a)] $\dom g_i+\dom l_i={\cal{G}}_i$, $i=1,\!...,m$; 
\item[(b)] ${\cal H}$ and ${\cal{G}}_i$ are 
finite-dimensional 
and there exists $x\in\ri\dom f$ such that $L_ix-r_i\in\ri\dom g_i+\ri\dom l_i$, $i=1,\!...,m$.
\end{enumerate}
\end{remark}

\subsection{A primal-dual algorithm of Douglas-Rachford-type with Tikhonov regularization terms}\label{sec52}

The problem that we investigate in this section reads as follows.

\begin{problem}\label{dr_p1}
Let $A:\h \rightrightarrows \h$ be a maximally monotone operator. Let $m$ be a strictly positive integer and for any 
$i = 1,...,m,$ let ${\cal G}_i$  be a real Hilbert space, $B_i,\,D_i:{\cal G }_i \rightrightarrows {\cal G}_i$ be maximally monotone operators and 
$L_i:{\cal H}\rightarrow$ ${\cal G}_i$ a nonzero linear continuous operator. The problem is to solve 
the primal inclusion
\begin{align}
	\label{dr_opt-p}
	\text{find }\bx \in \h \text{ such that } 0 \in A\bx + \sum_{i=1}^m L_i^* (B_i\Box D_i)(L_i \bx)
\end{align}
together with the dual inclusion
\begin{align}
	\label{dr_opt-d}
	\text{find }\bv_1 \in \g_1,\!...,\bv_m \in \g_m \text{ such that }(\exists x\in\h)\left\{
	\begin{array}{l}
		 - \sum_{i=1}^m L_i^*\bv_i \in Ax \\
		\!\!\bv_i \in \!\! (B_i \Box D_i)(L_ix), \,i=1,\!...,m.
	\end{array}
\right.
\end{align}
\end{problem}

Different to Problem \ref{pr1}, the operators $D_i, i=1,..., m$ are general maximally monotone operators, thus they will have to be addressed through their resolvents. This is why in this context a 
primal-dual algorithm relying on the Douglas-Rachford paradigm is more appropriate.

\begin{theorem}\label{dr-thm} In Problem \ref{dr_p1}, suppose that 
\begin{align}\label{dr_zin}
	0 \in \ran \bigg( A +  \sum_{i=1}^m L_i^*\circ(B_i\Box D_i)\circ L_i  \bigg).
\end{align}
Let $\tau, \sigma_i >0$, $i=1,\!...,m,$ be strictly positive numbers such that
$$\tau \sum_{i=1}^m \sigma_i \|L_i\|^2 < 4.$$ 
Let be the starting point  $(x_0,v_{1,0},...,v_{m,0})\in{\cal H}\times \g_1 \, ... \times \g_m$ and set:
	\begin{align}\label{dr_A1}
	  \left(\forall n\geq 0\right) \begin{array}{l} \left\lfloor \begin{array}{l}
		p_{1,n} = J_{\tau A}\left( \beta_nx_n  - \frac{\tau}{2}\beta_n \sum_{i=1}^m L_i^*v_{i,n} \right) \\
		w_{1,n} = 2p_{1,n} - \beta_nx_n \\
		\text{For }i=1,\!...,m  \\
				\left\lfloor \begin{array}{l}
					p_{2,i,n} = J_{\sigma_i B_i^{-1}}\left(\beta_nv_{i,n} +\frac{\sigma_i}{2} L_i w_{1,n} \right) \\
					w_{2,i,n} = 2 p_{2,i,n} - \beta_nv_{i,n} 
				\end{array} \right.\\
		z_{1,n} = w_{1,n} - \frac{\tau}{2} \sum_{i=1}^m L_i^* w_{2,i,n} \\
		x_{n+1} = \beta_nx_n + \lambda_n ( z_{1,n} - p_{1,n} ) \\
		\text{For }i=1,\!...,m  \\
				\left\lfloor \begin{array}{l}
					z_{2,i,n} = J_{\sigma_i D_i^{-1}}\left(w_{2,i,n} + \frac{\sigma_i}{2}L_i (2 z_{1,n} - w_{1,n}) \right) \\
					v_{i,n+1} = \beta_nv_{i,n} + \lambda_n (z_{2,i,n} - p_{2,i,n}),
				\end{array} \right. \\ \vspace{-4mm}
		\end{array}
		\right.
		\end{array}
	\end{align}
where
$(\lambda_n)_{n\geq 0}$ and $(\beta_n)_{n\geq 0}$ real sequences satisfying  the conditions: 

(i) $0<\beta_n\leq 1$ for any $n\geq 0$, $\lim_{n\rightarrow+\infty}\beta_n=1$, $\sum_{n\geq 0} (1-\beta_n)=+\infty$ and 
$\sum_{n \geq 1} |\beta_n-\beta_{n-1}|<+\infty$; 

(ii) $0<\lambda_n\leq 2$ for any $n\geq 0$, $\liminf_{n\rightarrow+\infty}\lambda_n>0$ and 
$\sum_{n\geq 1} |\lambda_n-\lambda_{n-1}|<+\infty$.

Then there exists an element 
$(\bx, \bv_1,\!...,\bv_m) \in \h \times \g_1 \, ... \times \g_m$ such that the following statements are true:
\begin{enumerate} 
	\item[(a)]\label{thm1.1} by setting 
	\begin{align*}
	\bp_1 &= J_{\tau A}\left( \bx - \frac{\tau}{2} \sum_{i=1}^m L_i^*\bv_{i}  \right), \\
	\bp_{2,i} &= J_{\sigma_i B_i^{-1}}\left(\bv_{i}+\frac{\sigma_i}{2} L_i (2\bp_1-\bx)  \right),\ i=1,\!...,m,
	\end{align*}
	the element $(\bp_1,\bp_{2,1},\!...,\bp_{2,m}) \in \h \times \g_1 \times \!... \times \g_m$ is a primal-dual solution to Problem \ref{dr_p1};
	\item[(b)]\label{thm1.3} $(x_n,v_{1,n},\!...,v_{m,n})$ converges strongly to $(\bx, \bv_1,\!...,\bv_m)$ as $n\rightarrow+\infty$;
	\item[(c)]\label{thm1.5} $(p_{1,n},p_{2,1,n},\!...,p_{2,m,n})$ and $(z_{1,n},z_{2,1,n},\!...,z_{2,m,n})$ 
	converge strongly to $(\bp_1,\bp_{2,1},\!...,\bp_{2,m})$ as $n\rightarrow+\infty$.                    
\end{enumerate} 
\end{theorem}

\begin{proof}
For the proof we use Theorem \ref{tikhonov-dr-operators} (see also \cite{b-h2}) in the same setting as in the proof of Theorem 
\ref{fb-pd-tikhonov}, namely, by considering $\fK = \h \times \g_1 \times \,... \times \g_m$ endowed with inner product and associated norm defined in \eqref{dr_def1.001}. 
Furthermore, we consider again the maximally monotone operator
\begin{align*}
	\f M : \fK \rightrightarrows \fK, \quad (x,v_1,\!...,v_m) \mapsto ( Ax,  B_1^{-1}v_1, \!...,  B_m^{-1}v_m),
\end{align*}
the linear continuous skew-symmetric operator
\begin{align*}
	\f S : \fK \rightarrow \fK, \quad (x,v_1,\!...,v_m) \mapsto \left(\sum_{i=1}^m L_i^* v_i, -L_1 x , \!..., - L_m x\right)
\end{align*}
and the (this time not necessarily single-valued) maximally monotone operator
\begin{align*}
	\f Q : \fK \rightrightarrows \fK, \quad (x,v_1,\!...,v_m) \mapsto \left(0, D_1^{-1} v_1, \!..., D_m^{-1} v_m\right).
\end{align*}
Since $\dom \f S = \fK$, both $\frac{1}{2}\f S+ \f Q$ and $\frac{1}{2}\f S + \f M$ are maximally monotone 
(see \cite[Corollary 24.4(i)]{bauschke-book}). Furthermore, 
$$\eqref{dr_zin} \Leftrightarrow \,\zer\left(\f M + \f S + \f Q\right) \neq \varnothing$$ 
and
\begin{align}
	\label{dr_optcond}
	\begin{aligned}
	&(x,v_1,\!...,v_m) \in \zer\left(\f M + \f S + \f Q\right) \\ 
	\Leftrightarrow &(x, v_1,\!...,v_m) \text{ is a primal-dual solution to Problem } \ref{dr_p1}.
	\end{aligned}
\end{align}
We introduce the linear continuous operator
\begin{align*}
	\f V : \fK \rightarrow \fK, \quad (x,v_1,\!...,v_m) \mapsto \left(\frac{x}{\tau} -\frac{1}{2} \sum_{i=1}^m L_i^* v_i, 
	\frac{v_1}{\sigma_1} - \frac{1}{2} L_1x , \!..., \frac{v_m}{\sigma_m} -\frac{1}{2} L_m x\right),
\end{align*}
which is self-adjoint and $\rho$-strongly positive (see also \cite{b-h2}), for 
$$\rho := \left(1-\frac{1}{2} \sqrt{\tau \sum_{i=1}^m \sigma_i \|L_i\|^2}\right) \min\left\{\frac{1}{\tau}, \frac{1}{\sigma_1}, 
\ldots, \frac{1}{\sigma_m} \right\}>0,$$
namely, the following inequality holds $$\< \f x,\f V\f x\>_{\fK}\geq \rho\|\f x\|_{\fK}^2 \ \forall \f x\in\fK.$$
Therefore, its inverse operator $\f V^{-1}$ exists and it fulfills $\|\f V^{-1}\|\leq \frac{1}{\rho}$.

The algorithmic scheme \eqref{dr_A1} is equivalent to
	\begin{align}\label{dr_A1.1}
	  (\forall n \geq 0) \ \!\begin{array}{l} \left\lfloor \begin{array}{l}
		\frac{\beta_nx_n- p_{1,n}}{\tau} - \frac{1}{2} \sum_{i=1}^m L_i^*(\beta_nv_{i,n}-p_{2,i,n}) \in \frac{1}{2} \sum_{i=1}^m L_i^*p_{2,i,n}+Ap_{1,n} \\
		w_{1,n} = 2p_{1,n} -\beta_n x_n  \\
		\text{For }i=1,\!...,m  \\
				\left\lfloor \begin{array}{l}
				 \frac{\beta_nv_{i,n} - p_{2,i,n}}{\sigma_i} 	- \frac{1}{2} L_i (\beta_nx_n - p_{1,n})  \in -\frac{1}{2}L_i p_{1,n} + B_i^{-1}p_{2,i,n} \\
					w_{2,i,n} = 2 p_{2,i,n} - \beta_nv_{i,n} \\
				\end{array} \right.\\
		\frac{w_{1,n}-z_{1,n}}{\tau} - \frac{1}{2}\sum_{i=1}^m L_i^* w_{2,i,n} = 0 \\
		x_{n+1} = \beta_nx_n + \lambda_n ( z_{1,n} - p_{1,n} ) \\
		\text{For }i=1,\!...,m  \\
				\left\lfloor \begin{array}{l}
					\frac{w_{2,i,n} - z_{2,i,n}}{\sigma_i} - \frac{1}{2}L_i(w_{1,n}-z_{1,n}) \in -\frac{1}{2}L_i z_{1,n} + D_i^{-1}z_{2,i,n} \\
					v_{i,n+1} =\beta_n v_{i,n} + \lambda_n (z_{2,i,n} - p_{2,i,n}). \\
				\end{array} \right. \\		
		\end{array}
		\right.
		\end{array}
	\end{align}
By considering for any $n\geq 0$ the notations 
$$\fx_n = (x_n,v_{1,n},\!...,v_{m,n}), \fy_n =(p_{1,n},p_{2,1,n},\!...,p_{2,m,n}) \ \mbox{and} \ \fz_n =(z_{1,n},z_{2,1,n},\!...,z_{2,m,n}),$$
the iterative scheme \eqref{dr_A1.1} can be written as
\begin{align}
	\label{dr_A1.2}
	\left(\forall n\geq 0\right)  \left\lfloor \begin{array}{l}
	\f V(\beta_n\fx_n - \fy_n ) \in \left(\frac{1}{2}\f S +\f M\right)\fy_n \\
	\f V(2 \fy_n -\beta_n \fx_n- \fz_n ) \in \left(\frac{1}{2}\f S +\f Q\right)\fz_n  \\
	\fx_{n+1} = \beta_n\fx_n + \lambda_n \left(\fz_n-\fy_n\right)
	\end{array}
	\right.
\end{align}
which is further equivalent to 
\begin{align}
	\label{dr_A1.3}
	\left(\forall n\geq 0\right)  \left\lfloor \begin{array}{l}
	\fy_n = \left(\id + \f V^{-1}(\frac{1}{2}\f S +\f M)\right)^{-1}\left(\beta_n\fx_n \right) \\
	\fz_n = \left(\id + \f V^{-1}(\frac{1}{2}\f S +\f Q)\right)^{-1}\left(2 \fy_n - \beta_n\fx_n \right) \\
	\fx_{n+1} = \beta_n\fx_n  + \lambda_n \left(\fz_n-\fy_n\right).
	\end{array}
	\right.
\end{align}
Let $\fK_{\f V}$ be the Hilbert space with inner product and norm defined, for $\fx,\fy \in \fK$, via
\begin{align}\label{dr_HSKV2}
 \< \fx,\fy \>_{\fK_{\f V}} = \< \fx, \f V \fy \>_{\fK} \text{ and } \|\fx\|_{\fK_{\f V}} = \sqrt{\< \fx, \f V \fx \>_{\fK}},
\end{align}
respectively. As the set-valued operators $\frac{1}{2}\f S+ \f M$ and $\frac{1}{2}\f S+ \f Q$ are maximally monotone on $\fK$, the operators
\begin{align}\label{dr_def1.1}
		\f B := \f V^{-1}\left(\frac{1}{2}\f S +\f M\right) \ \mbox{and} \ \f A:=\f V^{-1}\left(\frac{1}{2}\f S +\f Q\right)
\end{align}
are maximally monotone on $\fK_{\f V}$. Furthermore, since $\f V$ is self-adjoint and $\rho$-strongly positive, 
strong convergence in $\fK_{\f V}$ is equivalent with strong convergence in $\fK$.

Taking this into account, \eqref{dr_A1.3} becomes
\begin{align}
	\label{dr_A1.4}
	\left(\forall n\geq 0\right)  \left\lfloor \begin{array}{l}
	\fy_n = J_{\f B}\left(\beta_n\fx_n \right) \\
	\fz_n = J_{\f A}\left(2 \fy_n - \beta_n\fx_n \right) \\
	\fx_{n+1} = \beta_n\fx_n + \lambda_n \left(\fz_n-\fy_n\right),
	\end{array}
	\right.
\end{align}
which is the Douglas--Rachford algorithm formulated in Theorem \ref{tikhonov-dr-operators} in case $\gamma=1$ for determining the zeros of $\f A+\f B$. It is easy to see that
$$\zer(\f A+\f B) = \zer(\f V^{-1}\left(\f M+\f S + \f Q\right)) = \zer(\f M + \f S + \f Q).$$

By Theorem \ref{tikhonov-dr-operators} (a), there  exists $\fbx=(\ol x,\ol v_1,...,\ol v_m)\in\fix(R_{\f A}R_{\f B})$, such that 
$J_{\f B}\fbx\in\zer(\f A+\f B) = \zer(\f M + \f S + \f Q)$. The claim follows from Theorem \ref{tikhonov-dr-operators}, \eqref{dr_optcond} and by writing $J_{\f B}\fbx$ in terms of the resolvents of the operators involved in the expression of $\f B$. 
\end{proof}

We close this section by considering the variational case.

\begin{problem}\label{dr_p1_convex}
Let $f\in\Gamma({\cal H})$, $m$ be a strictly positive integer and for $i=1, ..., m$, let ${\cal G}_i$  be a real Hilbert space,
$g_i, l_i \in\Gamma({\cal G}_i)$ and $L_i:{\cal H}\rightarrow$ ${\cal G}_i$ a nonzero linear continuous operator. Consider the convex optimization problem
\begin{align}
	\label{dr_opt-mp}
	 \inf_{x \in \h}{\left\{f(x)+\sum_{i=1}^m (g_i \Box l_i)(L_ix) \right\}}
\end{align}
and its conjugate dual problem
\begin{align}
	\label{dr_opt-md}
	\sup_{(v_1,...,v_m) \in \g_1\times\,...\times\g_m}{\left\{-f^*\left(  - \sum_{i=1}^m L_i^*v_i\right) - 
	\sum_{i=1}^m \left( g_i^*(v_i) + l_i^*(v_i)  \right) \right\} }.
\end{align}
\end{problem}

\begin{corollary}\label{dr-thm-opt} In Problem \ref{dr_p1_convex}, suppose that 
\begin{align}\label{dr_zin-opt}
	0 \in \ran \bigg( \partial f +  \sum_{i=1}^m L_i^*\circ(\partial g_i\Box \partial l_i)\circ L_i  \bigg).
\end{align}
Let $\tau, \sigma_i >0$, $i=1,\!...,m,$ be strictly positive numbers such that
$$\tau \sum_{i=1}^m \sigma_i \|L_i\|^2 < 4.$$ 
Let be the starting point  $(x_0,v_{1,0},...,v_{m,0})\in{\cal H}\times \g_1 \, ... \times \g_m$ and set:

	\begin{align}\label{dr_A1-opt}
	  \left(\forall n\geq 0\right) \begin{array}{l} \left\lfloor \begin{array}{l}
		p_{1,n} = \prox_{\tau f}\left( \beta_nx_n  - \frac{\tau}{2}\beta_n \sum_{i=1}^m L_i^*v_{i,n} \right) \\
		w_{1,n} = 2p_{1,n} - \beta_nx_n \\
		\text{For }i=1,\!...,m  \\
				\left\lfloor \begin{array}{l}
					p_{2,i,n} =\prox_{\sigma_i g_i^*}\left(\beta_nv_{i,n} +\frac{\sigma_i}{2} L_i w_{1,n} \right) \\
					w_{2,i,n} = 2 p_{2,i,n} - \beta_nv_{i,n} 
				\end{array} \right.\\
		z_{1,n} = w_{1,n} - \frac{\tau}{2} \sum_{i=1}^m L_i^* w_{2,i,n} \\
		x_{n+1} = \beta_nx_n + \lambda_n ( z_{1,n} - p_{1,n} ) \\
		\text{For }i=1,\!...,m  \\
				\left\lfloor \begin{array}{l}
					z_{2,i,n} = \prox_{\sigma_i l_i^*}\left(w_{2,i,n} + \frac{\sigma_i}{2}L_i (2 z_{1,n} - w_{1,n}) \right) \\
					v_{i,n+1} = \beta_nv_{i,n} + \lambda_n (z_{2,i,n} - p_{2,i,n}),
				\end{array} \right. \\ \vspace{-4mm}
		\end{array}
		\right.
		\end{array}
	\end{align}
where $(\lambda_n)_{n\geq 0}$ and $(\beta_n)_{n\geq 0}$ are real sequences satisfying  the conditions: 

(i) $0<\beta_n\leq 1$ for any $n\geq 0$, $\lim_{n\rightarrow+\infty}\beta_n=1$, $\sum_{n\geq 0} (1-\beta_n)=+\infty$ and 
$\sum_{n \geq 1} |\beta_n-\beta_{n-1}|<+\infty$; 

(ii) $0<\lambda_n\leq 2$ for any $n\geq 0$, $\liminf_{n\rightarrow+\infty}\lambda_n>0$ and 
$\sum_{n\geq 1} |\lambda_n-\lambda_{n-1}|<+\infty$.

Then there exists an element 
$(\bx, \bv_1,\!...,\bv_m) \in \h \times \g_1 \, ... \times \g_m$ such that the following statements are true:
\begin{enumerate} 
	\item[(a)]\label{thm1.1-opt} by setting 
	\begin{align*}
	\bp_1 &= \prox\nolimits_{\tau f}\left( \bx - \frac{\tau}{2} \sum_{i=1}^m L_i^*\bv_{i}  \right), \\
	\bp_{2,i} &= \prox\nolimits_{\sigma_i g_i^*}\left(\bv_{i}+\frac{\sigma_i}{2} L_i (2\bp_1-\bx)  \right),\ i=1,\!...,m,
	\end{align*}
	the element $(\bp_1,\bp_{2,1},\!...,\bp_{2,m}) \in \h \times \g_1 \times \!... \times \g_m$ is a primal-dual solution to Problem \ref{dr_p1}, namely, 
	\begin{align}
	\label{dr_operator-proof-conditions-full-0}
	 - \sum_{i=1}^m L_i^*\bv_i \in \partial f(\bx) \ \mbox{and} \  \bv_i \in (\partial g_i \Box \partial l_i)(L_i\bx), \,i=1,\!...,m,
        \end{align}
	hence $\bp_1$ is an optimal solution to \eqref{dr_opt-mp} and $(\bp_{2,1},\!...,\bp_{2,m})$ is an optimal solution to \eqref{dr_opt-md};
	\item[(b)]\label{thm1.3-opt} $(x_n,v_{1,n},\!...,v_{m,n})$ converges strongly to $(\bx, \bv_1,\!...,\bv_m)$ as $n\rightarrow+\infty$;
	\item[(c)]\label{thm1.5-opt} $(p_{1,n},p_{2,1,n},\!...,p_{2,m,n})$ and $(z_{1,n},z_{2,1,n},\!...,z_{2,m,n})$ 
	converge strongly to $(\bp_1,\bp_{2,1},\!...,\bp_{2,m})$ as $n\rightarrow+\infty$.                    
\end{enumerate} 
\end{corollary}

\begin{remark}\rm The hypothesis \eqref{dr_zin-opt} in Theorem \ref{dr-thm-opt} is fulfilled, if the primal problem \eqref{dr_opt-mp} has an optimal solution, the regularity condition \eqref{reg-cond} holds and 
$$0\in\sqri(\dom g_i^*-\dom l_i^*) \mbox{ for }i=1,...,m.$$
According to \cite[Proposition 15.7]{bauschke-book}, the latter also guarantees that $g_i\Box l_i\in \Gamma(\g_i)$, $i=1,...,m$. 
\end{remark}

\section{Numerical experiments: applications to the split feasibility problem}\label{sec6}

Let $\mathcal{H}$ and $\mathcal{G}$ be real Hilbert spaces and $L : \mathcal{H} \to \mathcal{G}$ a bounded linear operator. Let $C$ and $Q$ be nonempty, closed and convex subsets of $\mathcal{H}$ and $\mathcal{G}$, respectively. The \textit{split feasibility problem} (SFP) searches a point $x \in \mathcal{H}$ 
with the property 
\begin{align}\label{SFP}
x \in C \text{ and } Lx \in Q.
\end{align}
The (SFP) was originally introduced by Censor and Elfving \cite{CeEl} for solving inverse problems in the context of phase retrieval, medical image reconstruction and intensity modulated radiation therapy. 

We show that the strongly convergent primal-dual algorithms which we have investigated in Section \ref{sec5} are excellently suited to solve the (SFP), especially in the case when infinite dimensional Hilbert spaces are involved. For this purpose, we note that problem $\eqref{SFP}$ can be written equivalently in the form 
\begin{align}\label{min1}
\min_{x \in \mathcal{H}} \left\{ \delta_C(x) + \delta_Q(Lx) \right\}, 
\end{align} 
where 
\begin{align*}
\delta_S(x) := \left\{\begin{array}{ll} 0, & \text{if } x \in S \\
+ \infty, & \text{else} \end{array}\right. 
\end{align*} 
denotes the indicator function of a subset $S$ of a Hilbert space.
Another alternative way to write problem \eqref{SFP} is as the minimization problem
\begin{align}\label{min2}
\min_{x \in \mathcal{H}} \left\{ \frac{1}{2} d_C^2(x) + \delta_Q(Lx) \right\},
\end{align}
where $d_C(x) = \inf_{u \in \mathcal{C}} \|x - u\|$ denotes the distance of the point $x \in \mathcal{H}$ to the set $C$. 
Both optimization problems are special instances of Problem \ref{pr}, consequently, we will solve them by making use of the algorithm stated in 
Corollary \ref{fb-pd-tikhonov-opt}. 

The optimization problem \eqref{min1} can be stated in the framework of Problem \ref{pr} by taking $f= \delta_C$, $m=1$, $g_1 = \delta_Q$, $l_1 = \delta_{\{0\}}$, $L_1=L$ and $h=0$. 
The iterative scheme stated in Corollary \ref{fb-pd-tikhonov-opt} reads:
\begin{align}\label{alg1l2}
	  \left(\forall n\geq 0\right) \begin{array}{l} \left\lfloor \begin{array}{l}
p_n = P_C(\beta_n x_n - \tau \beta_n L^\ast v_n) \\
x_{n+1} = \beta_n x_n + \lambda_n (p_n - \beta_n x_n) \\
q_n = \beta_n v_n + \sigma L(2p_n - \beta_n x_n) - \sigma P_Q(\sigma^{-1} \beta_n v_n + L(2p_n - \beta_n x_n)) \\
v_{n+1} = \beta_n v_n + \lambda_n (q_n - \beta_n v_n).
		\end{array}
		\right.
		\end{array}
\end{align}

The optimization problem \eqref{min1} can be stated in the framework of Problem \ref{pr} by taking $f= 0$, $m=1$, $g_1 = \delta_Q$, $l_1 = \delta_{\{0\}}$, $L_1=L$ and $h= \frac{1}{2} d_C^2$. 
Noticing that $\nabla \left(\frac{1}{2} d_C^2\right) = \operatorname{Id} - P_C$, the iterative scheme stated in Corollary \ref{fb-pd-tikhonov-opt} reads:
\begin{align}\label{alg2l2}
	  \left(\forall n\geq 0\right) \begin{array}{l} \left\lfloor \begin{array}{l}
p_n = \beta_n x_n - \tau ( \beta_n L^\ast v_n + \beta_n x_n - P_C(\beta_n x_n) ) \\
x_{n+1} = \beta_n x_n + \lambda_n (p_n - \beta_n x_n) \\
q_n = \beta_n v_n + \sigma L(2p_n - \beta_n x_n) - \sigma P_Q(\sigma^{-1} \beta_n v_n + L(2p_n - \beta_n x_n)) \\
v_{n+1} = \beta_n v_n + \lambda_n (q_n - \beta_n v_n).
		\end{array}
		\right.
		\end{array}
\end{align}

For the numerical example we consider the following setup: let $\mathcal{H} = \mathcal{G} = L^2([0,2\pi]) := 
\left\{ f : [0,2\pi] \to \mathbb{R} : \int_{0}^{2\pi} |f(t)|^2 dt < + \infty \right\}$ equipped with the scalar product 
$\langle f, g \rangle := \int_{0}^{2\pi} f(t)g(t) dt$ and the associated norm $\|f\| := \left( \int_{0}^{2\pi} |f(t)|^2 dt \right)^{1/2}$ 
for all $f,g \in L^2([0,2\pi])$. The sets
\[C := \left\{ x \in L^2([0,2\pi]) : \int_{0}^{2\pi} x(t) dt \leq 1 \right\} \] and \[Q := \left\{ x \in L^2([0,2\pi]) : \int_{0}^{2\pi} |x(t) - \sin(t) |^2 dt \leq 16 \right\} \] 
are nonempty, closed and convex subsets of $L^2([0,2\pi])$. Notice that $C=\{x \in L^2([0,2\pi]) :\langle x,u\rangle\leq 1\}$ and 
$Q=\{x \in L^2([0,2\pi]) :\|x-f\|\leq 4\}$, where $u:[0,2\pi]\to\R, u(t)=1$ for all $t\in [0,2\pi]$, and 
$f:[0,2\pi]\to\R, f(t)=\sin t$ for all $t\in [0,2\pi]$. 
We define the linear continuous operator $L : L^2([0,2\pi]) \to L^2([0,2\pi])$ as
\begin{align*}
(Lx)(t) := \left(\int_{0}^{2\pi} x(s) ds\right) \cdot t. 
\end{align*}
Let $x, y \in  L^2([0,2\pi])$. Then
\begin{align*}
\langle Lx, y \rangle = \int_{0}^{2\pi} Lx(t) y(t) dt = \int_{0}^{2\pi} x(s) \int_{0}^{2\pi} t y(t) dt ~ds,
\end{align*}
and therefore 
\begin{align*}
(L^\ast y)(t) = \int_{0}^{2\pi} s y(s) ds.
\end{align*} 	
Furthermore, by using the Cauchy-Schwarz inequality yields 
\begin{align*}
\|L\|^2 &= \sup_{ \|x\| = 1 } \int_{0}^{2\pi} Tx(t)^2 dt = \sup_{ \|x\| = 1 } \int_{0}^{2\pi} t^2 dt \left( \int_{0}^{2 \pi} x(s) ds \right)^2 \\ 
&\leq \sup_{ \|x\| = 1 } \frac{8}{3} \pi^3 \int_{0}^{2 \pi} x(s)^2 ds \cdot 2\pi = \frac{16}{3} \pi^4.
\end{align*}
The projection onto the set $C$ can be computed as (see \cite[Example~28.16]{bauschke-book}) \[P_C(x) = \left\{\begin{array}{ll} \frac{1 - \int_{0}^{2\pi} x(t) dt}{2\pi} + x, & \text{if} \ \int_{0}^{2\pi} x(t) dt > 1 \\[3pt]
x, & \text{else}. \end{array}\right.\] 
On the other hand $P_Q$ is given by \cite[Example~28.10]{bauschke-book}
\[P_Q(x) = \left\{\begin{array}{ll} \sin + \frac{4(x - \sin)}{(\int_{0}^{2\pi} |x(t) - \sin(t)|^2 dt )^{1/2}}, & \text{if} \ \int_{0}^{2\pi} |x(t) - \sin(t)|^2 dt > 16 \\[3pt]
x, & \text{else}. \end{array}\right. \]

\begin{table}\label{table1}
	\begin{tabular}{lccc}
		\hline \hline
	    & &  \multicolumn{2}{c}{Number of iterations}         \\ \hline
		$x_0$ & $v_0$  & $\beta_n = 1$  &$\beta_n = 1 - \frac{1}{n+1}$  \\
		\hline
		$\frac{t^2}{10}$        & $\frac{t^2}{10}$             & 13  & 1   \\
		$\frac{t^2}{10}$        & $\frac{1}{2}e^t$             & 20 &  11  \\
		$\frac{t^2}{10}$        & $e^t + \frac{t^2}{24}$       & 21  & 12  \\
		$\frac{1}{2} e^t$       & $\frac{t^2}{10}$             & $>$ 150 &  11 \\
		$\frac{1}{2} e^t$       & $\frac{1}{2}e^t$             & 20 &  12 \\
		$\frac{1}{2} e^t$       & $e^t + \frac{t^2}{24}$       & 21 &  13 \\
		$e^t +\frac{t^2}{24}$   & $\frac{t^2}{10}$             & $>$ 150 &  15  \\
		$e^t +\frac{t^2}{24}$   & $\frac{1}{2}e^t$             & 20 &  13  \\
		$e^t +\frac{t^2}{24}$   & $e^t + \frac{t^2}{24}$       & 21 &  13  \\
		\hline \hline 
	\end{tabular} 	
\caption{Comparison of the variants without and with Tikhonov regularization terms of the primal-dual algorithm \eqref{alg1l2}, for different starting values and step sizes $\tau = 0.1$ and $\sigma = 0.01$.}
\end{table}

We implemented the algorithms \eqref{alg1l2} and \eqref{alg2l2} in MATLAB used symbolic computation for generating the sequences of iterates. One can easily notice that, in this particular setting, the split feasibility problem and, consequently, both addressed optimization problems are solvable. The numerical experiments confirmed that the primal-dual algorithms involving Tikhonov regularization terms outperform the ones without Tikhonov regularization terms, which correspond to the case when $\beta_n=1$ for every $n \geq 0$ and for which is known that they weakly convergence to a primal-dual solution of the corresponding KKT system of optimality conditions.

\begin{table}\label{table2}
	\begin{tabular}{lccc}
		\hline \hline
		& &  \multicolumn{2}{c}{Number of iterations}         \\ \hline
		$x_0$ & $v_0$  & $\beta_n = 1$  &$\beta_n = 1 - \frac{1}{n+1}$  \\
		\hline
		$\frac{t^2}{10}$        & $\frac{t^2}{10}$             & 24  & 1  \\
		$\frac{t^2}{10}$        & $\frac{1}{2}e^t$             & 46 &  10   \\
		$\frac{t^2}{10}$        & $e^t + \frac{t^2}{24}$       & 46  & 10  \\
		$\frac{1}{2} e^t$       & $\frac{t^2}{10}$             & 30 &  6 \\
		$\frac{1}{2} e^t$       & $\frac{1}{2}e^t$             & 24 &  11 \\
		$\frac{1}{2} e^t$       & $e^t + \frac{t^2}{24}$       & 35 &  21  \\
		$e^t +\frac{t^2}{24}$   & $\frac{t^2}{10}$             & 32 &  6  \\
		$e^t +\frac{t^2}{24}$   & $\frac{1}{2}e^t$             & 36 &  12  \\
		$e^t +\frac{t^2}{24}$   & $e^t + \frac{t^2}{24}$       & 24 &  11   \\
		\hline \hline 
	\end{tabular} 	
\caption{Comparison of the variants without and with Tikhonov regularization terms of the primal-dual algorithm \eqref{alg2l2}, for different starting values and step sizes $\tau = 0.1$ and $\sigma = 0.01$.}
\end{table}

In the two tables above  we present the numbers of iterations needed by the two algorithms to approach a solution of the split feasibility problem $(SFP)$. We consider as stopping criterion
$$E(x_n) := \frac{1}{2} ||P_C(x_n) - x_n||^2 + \frac{1}{2} ||P_Q(Lx_n) - Lx_n||^2 \leq 10^{-3},$$
while taking as relaxation variables $\lambda_n = 0.4$ for every $n \geq 0$ and as Tikhonov regularization parameters $\beta_n := 1 - \frac{1}{n+1}$ for every $n \geq 0$, respectively, $\beta_n=1$ for every $n \geq 0$ for the variants without Tikhonov regularization terms.

\section{Further work}\label{sec7}

We point out some directions of research related to proximal methods with Tikhonov regularization terms, which merit to be addressed starting from the investigations made in this paper:

\begin{enumerate}
 \item To consider in the numerical algorithms of type forward-backward and Douglas-Rachford proposed in this paper {\it dynamic step sizes}, which are known to increase the flexibility of the algorithms. This  can be for instance done by formulating first a Krasnosel'ski\u{\i}--Mann--type algorithm for determining an element in the intersection of the sets of fixed points of a family of nonexpansive operators $(T_k)_{k\geq 0}$. Suitable choices of the operators in this family, in the spirit of the investigations we made in the sections \eqref{sec3} and \eqref{sec4}, can lead to iterative algorithms with dynamic step sizes (see  \cite{combettes}). 

\item To employ in the iterative schemes discussed in this article {\it inertial and memory effects}, which are known to contribute to the acceleration of the convergence behaviour of the algorithms.

\item To translate the proposed numerical methods to a framework that goes {\it beyond the Hilbert space setting} considered in this article, in order to allow applications where functional spaces are involved. We refer the reader to \cite{xu2002, sahu-ansari-yao2015, sahu-yao2011} and the 
 references therein for tools and techniques which allow to prove convergence statements in Banach spaces with some particular underlying geometric structures.
\end{enumerate}

\end{document}